\documentclass[reqno]{amsart}

\usepackage{bm}
\usepackage{amsmath}
\usepackage{amsfonts}
\usepackage{amssymb}
\usepackage{amsthm}
\usepackage{graphicx}
\usepackage{color}
\usepackage{enumitem,enumerate,mathrsfs,comment}

\theoremstyle{plain}

\newcommand{\mc}{\mathcal}
\newcommand{\mb}{\mathbb}

\newcommand{\E}{\mathbb{E}}
\newcommand{\R}{\mathbb{R}}

\newcommand{\I}{\textbf{1}_}

\renewcommand{\l}{\left}
\renewcommand{\r}{\right}

\theoremstyle{definition}
\newtheorem{defn}{Definition}[section]

\theoremstyle{remark}
\newtheorem{rmk}[defn]{Remark}

\theoremstyle{plain}
\newtheorem{lem}[defn]{Lemma}
\newtheorem{thm}[defn]{Theorem}
\newtheorem{cor}[defn]{Corollary}

\newtheorem{prop}[defn]{Proposition}
\numberwithin{equation}{section}

\begin{document}

\title[Explicit solutions to BSDEs and applications to utility maximization]{Explicit solutions to quadratic BSDEs and applications to utility maximization in multivariate affine stochastic volatility models}
\author[Anja Richter]{Anja Richter}
\address{ETH Z\"urich, Departement Mathematik, R\"amistrasse 101, 8092 Z\"urich, Switzerland}
\email{anja.richter@math.ethz.ch
}
\begin{abstract}
Over the past few years quadratic Backward Stochastic Differential Equations (BSDEs) have been a popular field of research. However there are only very few examples where explicit solutions for these equations are known. In this paper we consider a class of quadratic BSDEs involving affine processes and show that their solution can be reduced to solving a system of generalized Riccati ordinary differential equations. In other words we introduce a rich and flexible class of quadratic BSDEs which are analytically tractable, i.e.~explicit up to the solution of an ODE. Our results also provide analytically tractable solutions to the problem of utility maximization and indifference pricing in multivariate affine stochastic volatility models. This generalizes univariate results of Kallsen and Muhle-Karbe \cite{km08} and some results in the multivariate setting of Leippold and Trojani \cite{lt10} by establishing the full picture in the multivariate affine jump-diffusion setting. In particular we calculate the interesting quantity of the power utility indifference value of change of numeraire. Explicit examples in the Heston, Barndorff-Nielsen-Shephard and multivariate Heston setting are calculated.
\end{abstract}
\thanks{The author thanks Josef Teichmann and Peter Imkeller for their helpful comments.
The main part of this research can also be found in the author's PhD thesis. }

\maketitle 

\allowdisplaybreaks

\section{Introduction}

Since Bismut \cite{b73} introduced linear BSDEs in the context of Pontryagin's maximum principle, they have been intensively studied.
Their popularity stems from the fact that they can be applied to many different areas, e.g. in the study of properties of partial differential equations (PDEs), see Briand and Confortola \cite{BriandConfortola} and N'Zi et al. \cite{NziOuknineSulem}.
BSDEs also appear in many fields of mathematical finance, see El Karoui et al. 
\cite{ElKarouiPengQuenez} or more recently El Karoui and Hamad\`ene \cite{kh09} for an overview.
Problems such as pricing and hedging of European options (compare \cite{ElKarouiPengQuenez}), stochastic recursive utility (see \cite{de92}), utility maximization problems (e.g. \cite{HuImkellerMuller}) and risk measures (e.g. \cite{bk05,p97}) have been tackled using BSDE techniques.

For linear BSDEs it is possible under certain integrability and boundedness conditions to describe the first component of a solution as conditional expectation, compare \cite{ElKarouiPengQuenez} Section 2.
This is already not possible anymore in the Lipschitz case and therefore solutions and their properties can mostly be studied numerically.
In the meantime there is a huge literature on numerics for Lipschitz BSDEs, see \cite{glw05,bt04,bd07} amongst many others.
The present work focusses on BSDEs with drivers of quadratic growth which were first investigated by Kobylanski \cite{k00} in a Brownian setting and later extended to a continuous martingale setting by Morlais in \cite{Morlais1}.
Imposing certain growth and Lipschitz conditions on the generator and assuming bounded terminal conditions existence and uniqueness of quadratic growth BSDEs are guaranteed.
However there exists much less research on numerics for quadratic BSDEs and only very few examples where an explicit solution is known.

Motivated by this lack of examples and by the most important applications of this theory  
we analyze conditions under which one can find explicit solutions to a class of quadratic growth BSDEs. Our setting is as follows. We consider a forward affine process valued in $S_d^+$, the cone of positive semidefinite $d \times d$ matrices, and analyse BSDEs whose terminal condition and generator depend on this forward process. The main question from a BSDE-point of view addressed here is:
Which structural conditions on the terminal condition and the generator are needed (e.g. linear, affine, quadratic) to allow us to solve the BSDE explicitly?

We have chosen the forward process to be an affine process $X$ on $S_d^+$. As Kallsen and Muhle-Karbe \cite{km08} in the univariate case we can relate systems with $S_d^+$-valued forward processes to multivariate, realistically modelled utility optimization problems. Notice that affine processes have found a growing interest in the literature due to their analytic tractability which stems from the affine transform formula
\begin{align*}
  \E \l[ e^{-\operatorname{Tr}(X_t u)} \r]&=\exp(-\phi(t,u)-\operatorname{Tr}(\psi(t,u) X_0)),
\end{align*}
for all $t \in [0,T]$ and $u \in S_d^+$.
The functions $\phi$ and $\psi$ solve a system of generalized Riccati ordinary differential equations (ODEs), which are specified via the model parameters. Affine processes have been applied in various fields in mathematical finance such as the theory of term structure of interest rates, option pricing in stochastic volatility models and credit risk, see e.g. \cite{cir85,ds00,bns01,d05} and the references therein. Note that it is not necessary to look at matrix-valued affine processes as we do, but we could have equally chosen $\R^n_+ \times \R^m$, which was characterized by Duffie et al. \cite{dfs03}, or $ \R^n \times S_d^+$. For the sake of presentation we chose $ S_d^+ $ since this state space is complicated enough to make important pitfalls visible (e.g.~no infinite divisibility, no polyhedral property, etc, see \cite{cfmt09}), but still allows for simple notation. We emphasize that most of our results carry over to the general state space case.

The forward-backward system we consider consists of an affine process $X$ on $S_d^+$ and a BSDE whose terminal condition is an affine function of this process.
Moreover the generator is allowed to have a more involved structure including a dependence on $X$ and a quadratic dependence in the control process $Z$.
In Theorem \ref{thm:matrixBSDEsolution} we carry out which analytic form the generator and the terminal value of the BSDE need for the solution to be determined by a matrix ODE.
This ODE is a generalized Riccati ODE which may explode in finite time because of its quadratic term.
Therefore it is necessary to find conditions such that the ODE possesses a unique finite solution on the whole time interval $[0,T]$.

\medskip

We apply our results to the problem of maximizing expected utility of terminal wealth in multivariate affine stochastic volatility models.
This problem is typically approached either by stochastic control methods leading to Hamilton-Jacobi-Bellman equations or martingale methods which we will use here.
The martingale method has also been used in \cite{km08},  where the authors were able to solve the power utility optimization problem in several univariate affine models. They obtain the solution using semimartingale characteristics and represent the optimal strategies in terms of an opportunity process.
Using a combination of martingale methods and our results on explicit solutions of BSDEs we extend these ideas to higher dimensions.
In particular we derive 
 explicit results for power and exponential utility in multivariate extensions of the model of Heston \cite{h93} and the model of Barndorff-Nielsen and Shephard \cite{bns01}.
The multivariate results presented here are mostly new and provide a thorough extension of Fonseca et al. \cite{dig09}, where the power utility case in the multivariate Heston model is treated.
We also want to mention the work of Leippold and Trojani \cite{lt10} where in a multivariate setting the optimal strategy and the value function for power utility maximization is given. 
Their considerations are justified in our general affine setting.

A particularly interesting application of our findings is the
following: in the case of exponential utility we are able to provide
analytic expressions for the indifference prices of variance swaps, which is well-known and can also be found in the literature. In the case of power utility -- due to the additive structure of indifference prices -- one cannot find tractable expressions of those prices. However, the equally
interesting concept of \emph{indifference value of change of numeraire}
is again analytically tractable. The indifference value of change of
numeraire is the price one is willing to pay to swap one numeraire
with another one. This can have two applications: one is the case
where an institution actually bases their portfolio optimization, e.g., on fixed interest rates, even though interest rates are floating. The indifference value of change of numeraire is consequently the value of a swap
contract particularly designed for compensating this model
mispecification. The second one is a foreign exchange situation where
replacing one numeraire by another one influences the
optimal portfolio problem and therefore leads to an indifference
value. We can provide fully tractable formulas in all these cases, see
Section \ref{sec:heston}.

\medskip

In Section 2 we introduce necessary notation and collect several results by Cuchiero et al. \cite{cfmt09} who give a complete characterization of affine processes on $S_d^+$.
Chapter \ref{chap:explsol} studies explicit solutions to BSDEs and the connection with generalized matrix-valued Riccati ODEs.
The results are then applied to expected utility maximization in Section \ref{sec:stochvol}.

\section{Notation and characterization of affine processes}

We start with the notation we use subsequently.
We denote $\R_+=[0,\infty)$, $\R_-=(-\infty,0]$ and $\R_{++}=(0, \infty)$, $\R_{--}=(-\infty,0)$.
We write $\mathbb C$ for the space of complex numbers. 
The space $M_d$ stands for $d \times d$ matrices with real entries and $I_d$ is the $d \times d$-identity matrix.
$S_d$ is the space of symmetric $d \times d$-matrices equipped with the scalar product $\operatorname{\operatorname{Tr}}(xy)$.
This inner product naturally induces a norm $\parallel x\parallel= \sqrt{\operatorname{Tr}(xx)}$.
We write $S_d^+$ (or $S_d^-$) for the closed cone of symmetric $d \times d$ positive (or negative) semidefinite matrices and $S_d^{++}$ (or $S_d^{--}$) for the open cone of $d \times d$ positive (or negative) definite matrices.
By $\partial S_d^+ = S_d^+ \setminus S_d^{++}$ we denote the boundary of $S_d^+$.
The cones $S_d^+$ and $S_d^{++}$ induce a strict partial order relation on $S_d$.
We write $x \preceq y$ if $y-x \in S_d^+$ and $x \prec y$ if $y-x \in S_d^{++}$.
For $i,j=1, \ldots,d$, we also introduce the matrices $e^{ij} \in M_d$  with $e^{ij}_{ij}=1$ and the remaining entries being $0$.

The Borel $\sigma$-algebra on a space $U \subseteq S_d$ is denoted by $\mathcal B(U)$ and $bS_d^+$ refers to the set of bounded real-valued measurable functions $f$ on $S_d$. 
The vector space $\R^d$ is equipped with the Euclidean norm $|\cdot|$.
We work on the finite time interval $[0,T]$, where $T>0$ is fixed.

Let us consider time-homogeneous Markov processes $X$ with state space $S_d^+$ and semigroup $(P_t)_{t \in [0,T]}$ where
\begin{align*}
 P_t f(x) = \int_{S_d^+} f(\xi) p_t(x,d\xi),
\end{align*}
and $x \in S_d^+$, $f \in bS_d^+$ and $p_t$ a probability transition function.
We refer to \cite{rw94} for further details.
The process $X$ is not necessarily conservative.
To construct a conservative process $X$ we use the one-point compactification $S_d^+ \cup \Delta$ of $S_d^+$.
We correspondingly define
\begin{align*}
  p_t(x, \{\Delta\})= 1- p_t(x,S_d^+), \qquad p_t(\Delta,\{\Delta\})=1, 
\end{align*}
for all $t \in [0,T]$ and $x \in S_d^+$ with the convention that functions $f$ on $S_d^+$ are extended to $S_d^+ \cup \Delta$ by setting $f(\Delta)=0$ and $|\Delta|=\infty$.
Let us now define an affine process on $S_d^+$.

\begin{defn}[\cite{cfmt09} Definition 2.1]
\label{defn:affine}
 A Markov process $X$ is called affine if it satisfies the following conditions.
\begin{itemize}
 \item[(i)] $X$ is stochastically continuous, i.e. 
$$\lim_{s \to t} \int_{S_d^+} f(\xi)p_s(x,d\xi) = \int_{S_d^+} f(\xi) p_t(x, d\xi),$$ 
for all $f \in bS_d^+$ and every $t \in [0,T]$, $x \in S_d^+$.
\item[(ii)] The Laplace transform of $X$ depends in an exponential affine way on the initial state.
More precisely, there exist functions $\phi : [0,T] \times S_d^+ \to \R_+$ and $\psi : [0,T] \times S_d^+ \to S_d^+$ such that
\begin{align*}
 P_t e^{-\operatorname{Tr}(xu)}=\int_{S_d^+} e^{-\operatorname{Tr}(\xi u)} p_t(x,d\xi)=\exp(-\phi(t,u)-\operatorname{Tr}(\psi(t,u) x)),
\end{align*}
for all $t \in [0,T]$ and $u,x \in S_d^+$.
\end{itemize}
\end{defn}

In \cite{cfmt09} Theorem 2.4 the authors establish the Feller property for every affine process $X$ on $S_d^+$.
This permits a c\`adl\`ag representation for a given process $X$.
We may thus define the space 
$\Omega$ 
of all c\`adl\`ag paths $\omega: [0,T] \to S_d^+ \cup \Delta$ with $\omega(s)= \Delta$ whenever $\omega(t-)=\Delta$ or $\omega(t)=\Delta$ for $s > t$, $s,t \in [0,T]$.
 For every $x \in S_d^+$, $\mathbb P_x$ is the (unique) probability measure on $(\Omega, \bigvee_{t \in  [0,T]}\mathcal F^X_t )$ such that $\mathbb P_x (X_0=x)=1$, where $(\mathcal F^X_t)$ is the natural filtration generated by $X$.
We write $\mc F^{(x)}$ and respectively $(\mathcal F^{(x)}_t)$ for the completion of $\bigvee_{t \in  [0,T]}\mathcal F^X_t $ or $(\mathcal F^X_t)$  with respect to $\mathbb P_x$.
If we define
\begin{align*}
  \mathcal F_t = \bigcap_{x \in S_d^+} \mathcal F^{(x)}_t,\quad \mbox{and} \quad
 \mathcal F = \bigcap_{x \in S_d^+} \mathcal F^{(x)},
\end{align*}
then the filtration $(\mathcal F_t)$ is right continuous and $X$ is still a Markov process wrt.~$(\mathcal F_t)$ for $t \in  [0,T]$.
In the following we will write ``a.s.'' for ``$\mathbb P_x$-a.s. for all $x \in S_d^+$''. 
We call $X$ a semimartingale if $X_t \I{\{X_t \in S_d^+\}}$ is a semimartingale on $(\Omega, \mathcal F, (\mathcal F_t), \mathbb P_x )$ for all $x \in S_d^+$. 
For the definition of the characteristics of a semimartingale we refer to \cite{js87} Section II.2.
Note that semimartingale characteristics are always given wrt.~a truncation function $\chi: S_d \to S_d$ which is a continuous bounded function such that $\chi(\xi)=\xi$ for $\xi$ in a neighborhood of $0$.

\begin{defn}
\label{defn:admparset}
  We call $(\alpha,b, \beta^{ij},m,\mu,\iota,\gamma)$ an admissible parameter set associated to a truncation function $\chi$ if it satisfies the following conditions.
\begin{itemize}
  \item[(i)] The linear diffusion coefficient $\alpha$ belongs to the cone $S_d^+$.
  \item[(ii)] The constant drift term $b$ is such that $b \succeq (d-1) \alpha$.
  \item[(iii)] The constant jump term $m$ is a Borel measure on $\mathcal B( S_d^+ \setminus \{0\})$ satisfying
\begin{align}
\label{eqn:integrabilitym}
\int_{S_d^+ \setminus \{0\}} (\parallel \xi \parallel \wedge 1) m (d\xi) < \infty.
\end{align}
  \item[(iv)] The linear jump term consists of a $d \times d $-matrix $\mu=(\mu_{ij})$ of finite signed measures on $\mathcal B(S_d^+ \setminus \{0\})$ such that $\mu(E) \in S_d^+$ for all $E \in \mathcal B(S_d^+ \setminus \{0\})$.
The kernel
\begin{align}
\label{eqn:M(x,dxi)}
  M(x,d\xi)= \frac{\operatorname{Tr}(x \mu(d\xi))}{\parallel \xi \parallel^2 \wedge 1},
\end{align}
satisfies
\begin{align}
  \label{eqn:integrabilityM}
\int_{S_d^+ \setminus \{0\}} \operatorname{Tr}(\chi(\xi )u) M(x,d\xi) < \infty, \quad \mbox{ for all } x,u \in S_d^+ \mbox{ with } \operatorname{Tr}(xu)=0.
\end{align}
\item[(v)] The linear drift coefficient is composed of a family $(\beta^{ij})_{i,j=1,\ldots,d}$ of symmetric matrices with $\beta^{ij}=\beta^{ji} \in S_d$ for all $i,j=1,\ldots,d,$ and such that the linear map $B: S_d \to S_d$ with
\begin{align}
\label{eqn:B(x)}
 B(x) & = \sum_{i,j=1}^d \beta^{ij} x_{ij} ,
\end{align}
fulfills
\begin{align}
\label{eqn:lindrift}
 \operatorname{Tr}(B(x)u) - \int_{S_d^+ \setminus \{0\}} \operatorname{Tr}(\chi(\xi )u) M(x,d\xi) \geq 0 ,
\end{align}
for all $x,u \in S_d^+$ with $\operatorname{Tr}(xu)=0$.
\item[(vi)] The constant killing rate coefficient $\iota$ has values in $\R_+$.
\item[(vii)] The linear killing rate coefficient $\gamma$ has values in $S_d^+$.
\end{itemize}
\end{defn}

A full characterization of affine processes on $S_d^+$ is given in \cite{cfmt09}.

\begin{thm}[\cite{cfmt09} Theorem 2.4]
\label{thm:charaffineproc}
Let $X$ be an affine process on $S_d^+$. 
Then there exists an admissible parameter set $(\alpha,b,\beta^{ij},m,\mu, \iota, \gamma)$ wrt.~a truncation function $\chi$ such that 
the functions $\phi$ and $\psi$ from Definition \ref{defn:affine} (ii) solve the generalized Riccati ODE
\begin{align}
\label{eqn:odephi1}  
\frac{\partial \phi(t,u)}{\partial t}
&= \mathscr F(\psi(t,u)), \quad \phi(0,u)=0,
\\
\label{eqn:odepsi1}
  \frac{\partial \psi(t,u)}{\partial t}
&= \mathscr R(\psi(t,u)), \quad \psi(0,u)=u \in S_d^+,
\end{align}
with
\begin{align*}
  \mathscr F(u)&= \operatorname{Tr}(bu) + \iota - \int_{S_d^+ \setminus \{0\}} (e^{-\operatorname{Tr}(u\xi)} -1) m(d\xi),
\\
\mathscr R(u)&= -2u\alpha u + B^{*}(u)+\gamma-\int_{S_d^+ \setminus \{0\}} \frac{e^{-\operatorname{Tr}(u\xi)}-1+\operatorname{Tr}(\chi(\xi)u)}{\parallel \xi \parallel^2 \wedge 1} \mu(d\xi),
\end{align*}
where $B^*_{ij}(u)=\operatorname{Tr}(\beta^{ij}u)$ for $i,j=1, \ldots,d$.

Conversely, let $(\alpha, b, \beta^{ij},m,\mu,\iota,\gamma)$ be an admissible parameter set associated to a truncation function $\chi$.
Then there exists a unique affine process on $S_d^+$ 
and the condition of Definition \ref{defn:affine} (ii) holds for all $(t,u) \in [0,T] \times S_d^+$, where $\phi$ and $\psi$ are given by \eqref{eqn:odephi1} and \eqref{eqn:odepsi1}.
\end{thm}

Every conservative affine process on $S_d^+$ with killing rate coefficients $\iota=\gamma=0$ is a semimartingale.

\begin{thm}[\cite{cfmt09} Theorem 2.6]
Let $X$ be a conservative affine process on $S_d^+$ and $(\alpha,b, \beta^{ij},m,\mu,0,0)$ the related admissible parameter set associated to a truncation function $\chi$.
Then $X$ is a semimartingale whose characteristics $(D,A,\nu)$ with respect to $\chi$ are given by
\begin{align*}
D_t &= \int_0^t \left( b + \int_{S_d^+ \setminus \{0\}} \chi(\xi) m (d\xi) +B(X_s) \right) ds,
\\
  A_{t,ijkl} & = \int_0^t A_{ijkl}(X_s) ds,
\\
\nu([0,t],G) & = \int_0^t (m(G) + M(X_s,G))ds,
\end{align*}
for $i,j,k,l \in \{1, \ldots , d\}$, $t \in [0,T]$ and $G \in \mathcal B(S_d^+ \setminus \{0\})$.
The matrix $B$ is given by \eqref{eqn:B(x)}, $M$ by \eqref{eqn:M(x,dxi)} and $A_{ijkl}$ by
\begin{align}
\label{eqn:Aijkl}
A_{ijkl}(x)= x_{ik}\alpha_{jl}+ x_{il}\alpha_{jk} + x_{jk} \alpha_{il} + x_{jl} \alpha_{ik} , 
\end{align}
for all $i,j,k,l \in \{1, \ldots , d\}$, and $x \in S_d^+$.
Moreover there exists a $d \times d$ matrix of standard Brownian motions $W$ such that $X$ has the following canonical representation
\begin{align}
\label{representationofx}
 X_t & =  x + \int_0^t \sqrt{X_s} dW_s \Sigma + \int_0^t \Sigma^{\top} dW^{\top}_s \sqrt{X_s} 
\\ \nonumber
& \quad + \int_0^t \left( b + B(X_s) + \int_{S_d^+ \setminus \{0\} } \chi(\xi) m(d\xi)\right) ds 
\\  \nonumber
& \quad + \int_0^t \int_{S_d^+ \setminus \{0\} } \chi(\xi) \left( \mu^X (ds,d\xi) - \nu(ds,d\xi) \right)
\\ \nonumber
& \quad
+ \int_0^t \int_{S_d^+ \setminus \{0\} } \left( \xi - \chi(\xi)\right) \mu^X (ds,d\xi),
\end{align}
where $\Sigma \in M_d$ satisfies $\Sigma \Sigma^{\top}= \alpha$ and $\mu^X$ denotes the random measure associated with the jumps of $X$.
\end{thm}

 \begin{rmk}
The canonical representation \eqref{representationofx} follows via the canonical semimartingale representation (see \cite{js87} Theorem II.2.34) and the construction of a matrix-valued Brownian motion.
For the latter one has to find a matrix which coincides with the covariation of the affine process.

Note that the constant drift term of an affine semimartingale is independent of the truncation function $\chi$ while $\chi$ influences the linear drift coefficient $B$.
 \end{rmk}

From now on we fix a truncation function $\chi$, and then write ``admissible parameter set'' for ``admissible parameter set associated to truncation function $\chi$''.
The affine process $X$ with admissible parameter set $(\alpha,b, \beta^{ij},m,\mu,\iota, \gamma)$ is continuous if and only if $m$ and $\mu$ vanish, i.e. $(\alpha,b,\beta^{ij},0,0,\iota, \gamma)$.
Since we only consider affine semimartingales we write $(\alpha,b,\beta^{ij},m,\mu)$ for $(\alpha,b,\beta^{ij},m,\mu,0,0)$.

\bigskip

To fix ideas let us give an example of a matrix-valued affine processes, the Wishart processes.
These processes were first rigorously studied in \cite{b91} extending squares of matrix Ornstein-Uhlenbeck processes.
The dynamics of a Wishart process satisfies
\begin{align}
\label{eqn:wishart}
dX_t&= (b+ HX_t + X_tH^{\top})dt + \sqrt{X_t} dW_t \Sigma + \Sigma^{\top} dW^{\top}_t \sqrt{X_t},
\end{align}
where $b, H, \Sigma \in M_d$ and $W$ is a $d\times d$ matrix Brownian motion.
These processes have been widely used to model stochastic covariation, see e.g. \cite{bpt10}, \cite{dgt07} and \cite{dgt08}.
In order to obtain a well defnned matrix volatility process Bru required the constant drift part $b = k \Sigma^{\top}\Sigma$ for some $k > d-1$. 
Then $X$ has a Wishart distribution.
In the above notation the admissible parameter set for the Wishart process is $(\Sigma^{\top}\Sigma,b, \beta^{ij},0,0)$ with $B(x)=Hx+xH^{\top}$.

In contrast to affine processes on the state space $\R^m_+ \times \R^n$, which were fully characterized in \cite{dfs03}, and where the diffusion term consists of a constant and linear part, the diffusion term of an affine process on $S_d^+$ with admissible parameter set $(\alpha,b,\beta^{ij},m,\mu)$ only allows for a linear part of the specific form
\begin{align}
\label{eqn:A}
\sum_{i,j,k,l=1}^d u_{ij}A_{ijkl}(x)u_{kl}=4\operatorname{Tr}(x u \alpha u), \quad x,u \in S_d^+.
\end{align}
Note that the necessity and sufficiency of conditions $(ii)$ and $(v)$ in definition \ref{defn:admparset} was first shown in \cite{cfmt09}.
In particular formula \eqref{eqn:B(x)} allows for a more general form than $B(x)=Hx+xH^{\top}$, $x \in S_d^+$, compare also \cite{cfmt09}, Chapter 2.1.2.
\section{Explicit solutions of quadratic FBSDEs}
\label{chap:explsol}


In this section we examine how the solutions for a class of quadratic BSDEs can be reduced to solving ODEs.
In contrast to many existence results in the literature, e.g. \cite{pp90,k00,m08} and \cite{b06}, where the generator $f$ is usually required to satisfy certain Lipschitz and growth conditions, we suggest an analytic expression for $f$ which gives the problem extra structure.
Consider the following motivating example where such a form appears naturally.

Take the one-dimensional Heston model (compare \cite{h93}) for the dynamics of an asset $H$.
The stochastic logarithm $N$ of $H$ 
satisfies
\begin{align*}
  dN_t&= \eta R_t dt + \sqrt{R_t} dQ_t,
\\
  dR_t&= (b + \lambda R_t) dt + \sigma \sqrt{R_t} dW_t, \quad t \in [0,T].
\end{align*}
Here $R$ is the stochastic volatility process, $b, \sigma >0$, $\eta,\lambda \in \R$ are constants and $Q$, $W$ are two Brownian motions with constant correlation $\rho \in[-1,1]$.
The volatility process $R$ satisfies our definition of an affine process on $\R_+$ with admissible parameter set $(\frac{1}{4} \sigma^2,b,\lambda,0,0)$.
We study an investor who is interested in maximizing their expected utility from terminal wealth.
The investor's initial capital is denoted by $x \in \R$ and their trading strategies are deterministic functions $\pi$ of time, where $\pi(t)$ describes the amount of money invested in stock $H$ at time $t \in [0,T]$.
The wealth process $X^{x, \pi}$ for initial endowment $x$ and strategy $\pi$ is given by 
$$X^{x, \pi}_t=x+ \int_0^t \frac{\pi(s)}{H_s} dH_s =x + \int_0^t \pi(s) dN_s,$$ 
for $t \in [0,T]$.
We can solve the exponential utility maximization problem
\begin{align*}
  V(x)= \sup_{\pi} \E \left[ - \exp \left( - \gamma X_T^{x, \pi}\right)\right],
\quad x \in \R, \quad \gamma >0,
\end{align*}
by finding the generator $f$ of the BSDE
\begin{align}
\label{eqn:hest1}  
Y_t&=0 - \int_t^T Z_s dW_s + \int_t^T f(R_s,Z_s) ds,
\end{align}
such that the process $L^\pi_t=- \exp(-\gamma (X_t^{x, \pi} +Y_t))$, $t \in [0,T]$, is a supermartingale for all strategies $\pi$ and a martingale for a particular strategy $\pi^{opt}$.
The required generator can be shown (see Lemma \ref{lemsupermartexp}) to be
\begin{align}
\label{eqn:genex}
  f(r,z)= \frac{\gamma}{2} (\rho^2 - 1) z^2 + \frac{1}{2 \gamma^3} \eta^2 r -\frac{1}{\gamma} \eta \rho z \sqrt{r}, \quad (r,z) \in \R^2,
\end{align}
similarly to \cite{HuImkellerMuller} Theorem 7.
Notice that the generator is quadratic in the $z$-component.
In order to solve this BSDE we apply the It\^o formula to an affine function of $R$.
More precisely we make an affine ansatz for $Y_t=\Gamma(t)R_t + w(t)$, $t \in [0,T]$, where $\Gamma,w:[0,T] \to \R$ are differentiable functions.
This leads to
\begin{align}
\label{eqn:hest2} 
  \Gamma(t)R_t + w(t) 
&= \Gamma(T)R_T + w(T) 
- \int_t^T \Gamma(s) \sigma \sqrt{R_s} dW_s 
\\ \nonumber
& \quad 
- \int_t^T \l(\Gamma(s)(b+ \lambda R_s) + \frac{d \Gamma(s)}{ds} R_s + \frac{d w(s)}{ds} \r) ds.
\end{align}
It can be immediately read off the equation that $\Gamma$ and $w$ must satisfy 
\begin{align}
\label{eqn:hest3}
  \Gamma(T)=w(T)=0 \mbox{ and } Z_s=\Gamma(s) \sigma \sqrt{R_s}, \quad s \in [0,T].
\end{align}
The finite variation parts of \eqref{eqn:hest1} and \eqref{eqn:hest2} coincide if for all $s \in [0,T]$
\begin{align*}
  0&=f(R_s,Z_s) + \Gamma(s)(b+ \lambda R_s) + \frac{d \Gamma(s)}{ds} R_s + \frac{d w(s)}{ds}
\\
& = 
 \frac{\gamma}{2} (\rho^2 - 1) \Gamma^2(s) \sigma^2 R_s 
+ \frac{1}{2 \gamma^3} \eta^2 R_s 
-\frac{1}{\gamma} \eta \rho \Gamma(s) \sigma R_s 
+ \Gamma(s)(b+ \lambda R_s) + \frac{d \Gamma(s)}{ds} R_s + \frac{d w(s)}{ds},
\end{align*}
where we have used the equation for the generator $f$ and formula \eqref{eqn:hest3} for $Z$.
Equating coefficients this leads to an ODE of Riccati type
\begin{align*}
  -\frac{d \Gamma(t)}{d t} &= 
q \Gamma^2(t)  + l \Gamma(t) + c, \quad \Gamma(T)=0, \quad t \in [0,T],
\end{align*}
with constants
\begin{align*}
  q=\frac{\gamma}{2} \sigma^2 (\rho^2 - 1), \quad 
l=\lambda - \frac{1}{\gamma} \sigma \rho \eta, \quad
c=\frac{1}{2 \gamma^3} \eta^2,
\end{align*}
and an ODE of the simpler form
\begin{align*}
   -\frac{d w(t)}{d t}&= \Gamma(t)b    ,\quad w(T)=0, \quad t \in [0,T].
\end{align*}
Hence the solution of \eqref{eqn:hest1} is
\begin{align*}
  Y_t&= \Gamma(t) R_t + \int_t^T b \Gamma(s) ds 
\\
 Z_t&= \Gamma(t) \sigma \sqrt{R_t}, \quad t \in [0,T].
\end{align*}
Generally, Riccati ODEs have the property that their solution can blow up in finite time, however our model parameter choices admit a non-explosive solution.
In the one-dimensional case considered here we can even give an fully explicit solution, compare \cite{b01} Section 21.5.1.2.
More specifically we distinguish two different cases depending on the value of
\begin{align*}
  d=l^2-4qc = \left( \lambda - \frac{1}{\gamma} \sigma \eta \rho \right)^2 + \frac{1}{\gamma^2} \sigma^2 \eta^2(1- \rho^2) \geq 0.
\end{align*}
If $d>0$, then
\begin{align*}
 \Gamma(t)= -2c \frac{e^{\sqrt{d} (T-t)} -1}{e^{\sqrt{d} (T-t)}(l + \sqrt{d}) - l +\sqrt{d}}
, \quad t \in [0,T].
\end{align*}
If $d=0$, then $\rho=1$, $\lambda=\frac{1}{\gamma} \sigma \eta$ and hence
\begin{align*}
  \Gamma(t)=\frac{1}{2\gamma^3} \eta^2 (T-t), \quad t \in [0,T].
\end{align*}
In both cases the martingale property of $-\exp(- \gamma (X^{x,\pi^{opt}} +Y))$ then gives the value function and the optimal strategy
\begin{align*}
  V(x)&=-\exp \left(-\gamma \left(x+ \Gamma(0)R_0 + \int_0^T b\Gamma(s) ds \right) \right),
\\
  \pi^{opt}(t)&=\frac{1}{\gamma^2} \eta - \Gamma(t) \sigma \rho,
\end{align*}
for $x \in \R$, $t \in [0,T]$.

\bigskip

In the previous example the ansatz $Y_t=\Gamma(t)R_t + w(t)$ and the method of equating coefficients enabled us to reduce the solution of a BSDE to solving ODEs.
Now that we have seen how we exploit the affine structure in an one-dimensional example, we generalize this to BSDEs depending on affine processes on $S_d^+$ and even an additional process which has affine semimartingale characteristics with respect to the affine process.
The question is how general we are allowed to choose the generator and the terminal condition in order to still be able to apply the above method.

We associate the affine process $X$ with admissible parameter set $(\alpha,b, \beta^{ij},m,\mu)$ to a BSDE.
To allow for a more flexible financial modeling, especially in view of pricing of variance swaps, the BSDE will moreover depend on 
the matrix-valued process
\begin{align}
\label{eqn:hatN}  
dO_t =  \sigma(t) \sqrt{X_t}  d\hat Q_t  + \left( o_1 (t) + o_2 (t) X_t \right) dt, \quad t \in [0,T],
\end{align}
where $o_1, o_2 : [0,T] \to M_d$ and $ \sigma : [0,T] \to M_d$ are continuous functions.
Here the process $\hat Q$ denotes a $d \times d$ matrix-valued Brownian motion.
This process will enable us to calculate indifference prices and delta hedges for variance swaps, see e.g. Section \ref{sec:contexput}.

Our (nonstandard) real-valued BSDE will have the following form 
\begin{align}
\label{defBSDE}
Y_t & =  
F(X_T,O_T) 
- \int_t^T \operatorname{Tr}(Z_s^{\top} dW_s) 
- \int_t^T \operatorname{Tr}(\hat Z_s^{\top} d\hat Q_s) 
\\ \nonumber
& \quad
- \int_0^t \int_{S_d^+ \setminus \{0\} } K_s(\xi) \left( \mu^X (ds,d\xi) - \nu(ds,d\xi) \right)
\\
& \quad 
+ \int_t^T f(s,X_s,Y_s,Z_s, \hat Z_s ,K_s) ds,
\nonumber
\end{align}
for $t \in [0,T]$, where the terminal condition $F$ is allowed to depend on the affine process $X$ and on the process $O$.
Recall that $W$ is the Brownian motion of the underlying affine process $X$ and the generator is a deterministic Borel measurable function $f: [0,T] \times S_d^+ \times \R \times M_d \times M_d \times \R \to \R
$.

\begin{defn}
  A solution to BSDE \eqref{defBSDE} is a family of processes $(Y,Z, \hat Z,K) \in \R \times M_d \times M_d \times \R$ such that:
\begin{itemize}
  \item[(i)] The equation \eqref{defBSDE} is a.s. satisfied.
  \item[(ii)] The stochastic integrals $Z^{\top} \cdot W$, $\hat Z^{\top} \cdot \hat Q$ are well-defined.
  \item[(iii)] The integrability condition $\int_0^T |f(s,X_s,Y_s,Z_s, \hat Z_s ,K_s)| ds < \infty$ holds true.
  \item[(iv)] The mapping $t \mapsto Y_t$ is c\`adl\`ag.
  \item[(v)] The process $K$ satisfies $\int_0^T \int_{S_d^+ \setminus \{0\}} |K_t(\xi)| (m(d\xi)+M(X_t,d\xi))dt  < \infty.$
\end{itemize}
\end{defn}
For a certain class of generators and terminal conditions we can give the solution to the above BSDE in terms of matrix ODEs.
Suppose the terminal condition $F:S_d^+ \times M_d \to \R$ is affine, more precisely 
\begin{align}
\label{eqn:terminalcondition}
 F(x, o)=\operatorname{Tr}(ux)
+ \operatorname{Tr}(a o ) + v, \quad x \in S_d^+, \ o \in M_d,
\end{align}
where $u \in S_d$, 
$ a \in M_d$ and $v \in \R$.
Let us define the set $L^0$ as the space which contains all functions $k:S_d^+ \to \R$.
The class of generators $f$ is more involved, more precisely the generator $f:[0,T] \times S_d^+ \times \R \times M_d \times M_d \times L^0 \to \R$ is allowed to have the following form
\begin{align}
\label{matrixgenerator}
& f(t,x,y,z,\hat z, k)
\\ \nonumber
&= 
\operatorname{Tr}(z c_{zz}(t) z^{\top}) 
+ \operatorname{Tr}( z c_{z\sqrt{x}}(t) \sqrt{x})  
+  \operatorname{Tr}( c_x(t) x)
+ c_y(t)y 
+c_t(t) 
\\ \nonumber
& \quad
+ \operatorname{Tr}(\hat z c_{\hat z \hat z}(t) \hat z^{\top})
+ \operatorname{Tr}(\hat z c_{\hat z z}(t) z^{\top})
+\operatorname{Tr}(\hat z c_{\hat z \sqrt{x}}(t) \sqrt{x})
\\ \nonumber
& \quad  
+ \int_{S_d^+ \setminus \{0\}} g_M(t,k(\xi)) M(x,d\xi)
\\ \nonumber
& \quad 
+ \int_{S_d^+ \setminus \{0\}}
\Big( \operatorname{Tr}( z g_{z\sqrt{x}}(t,k(\xi)) \sqrt{x}) 
+\operatorname{Tr}(x g_{x}(t,k(\xi))) 
+ g_t(t,k(\xi)) 
+  y g_{y}(t,k(\xi) 
\Big) m(d\xi)  \nonumber
\\ \nonumber
& \quad
+  \int_{S_d^+ \setminus \{0\}} \left(
\hat z g_{\hat z \hat z}(t,k(\xi)) \hat z^{\top})
+ \operatorname{Tr}(\hat z g_{\hat z z}(t,k(\xi)) z^{\top})
+ \operatorname{Tr}(\hat z g_{\hat z \sqrt{x}}(t,k(\xi)) \sqrt{x})
\right) m(d\xi),
\end{align}
for all $(t,x,y,z,\hat z,k) \in [0,T] \times S_d^+ \times \R \times M_d \times M_d \times L^0$.
In the above,
\begin{align*}
 c_{zz},c_{z \sqrt{x}},c_{x}, c_{\hat z \hat z},c_{\hat z z},c_{\hat z \sqrt x}: [0,T] &\to M_d,
\\
 c_t,c_y:[0,T] &\to \R,
\end{align*}
are continuous functions and $g_M:[0,T] \times L^0 \to \R$ is an $M(x,d\xi)$-integrable function, $x \in S_d^+$, which is continuous in time.
Finally
\begin{align*}
 g_{z \sqrt{x}},g_{x}, g_{\hat z \hat z},g_{\hat z z},g_{\hat z \sqrt x}: [0,T] \times L^0 &\to M_d, 
\\
 g_t,g_y
:[0,T] \times L^0 &\to \R,
\end{align*}
are $m(d\xi)$-integrable functions which are also assumed continuous in time.

\begin{rmk}
We restrict ourselves to the case where $f$ does not depend explicitly on $O$ since the structure of $f$ is already quite involved.
If this were not the case we would derive further coupled ODEs. 
However the terminal value $F$ depends affine on $O$.
We could allow for jumps in $O$ provided those jumps have affine characteristics in $X$, but do not for reasons of brevity.
It is possible to only consider functional forms of time
$\sigma(\cdot)$ rather than 

$\sigma(\cdot)\sqrt X$ in
\eqref{eqn:hatN}. 
\end{rmk}

We can now give the main theorem which describes the explicit form of the solution processes $(Y,Z, \hat Z,K)$ in terms of the solution to a system of generalized Riccati equations.
\begin{thm}
\label{thm:matrixBSDEsolution}
Let $X$ be an affine semimartingale on $S_d^+$ associated to the admissible parameter set $(\alpha,b,\beta^{ij},m,\mu)$ such that 
\begin{align}
  \label{eqn:integrability1}
\int_{\parallel \xi \parallel > 1} \parallel \xi \parallel (m(d\xi)+M(x,d\xi)) < \infty, \quad x \in S_d^+.
\end{align}
Furthermore suppose that there exists a unique solution $\Gamma(\cdot,u):[0,T] \to S_d$ to the generalized Riccati ODE
\begin{align}
\label{eqn:matrixode}
  - \frac{ \partial \Gamma(t,u)} {\partial t} = \theta(t, \Gamma(t,u)) , \quad \Gamma(T,u)=u,
\end{align}
with
\begin{align}
\label{eqn:matrixodetheta}
\theta(t,u)
&=
4 u \Sigma^{\top} c_{zz}(t) \Sigma u 
+ \mathscr L(t) u 
+ u \mathscr L^{\top}(t) 
+ B^*(u)
+ \mathscr C(t)
\\ \nonumber
& \quad 
+ \int_{S_d^+ \setminus \{0\}} 
\frac{\operatorname{Tr}\left(u(\xi-\chi(\xi))\right) + g_M(t,\operatorname{Tr}(u\xi)) }{\parallel \xi \parallel^2 \wedge 1} \mu(d\xi)
\\ \nonumber
& \quad 
+ \int_{S_d^+ \setminus \{0\}} 
\biggl(
 u \Sigma^{\top} g_{z \sqrt{x}}(t,\operatorname{Tr}(u\xi)) 
+ g_{z \sqrt{x}}^{\top}(t,\operatorname{Tr}(u\xi)) \Sigma u 
+ u g_y(t,\operatorname{Tr}(u\xi))
+ g_x(t,\operatorname{Tr}(u\xi)) 
\\ \nonumber
& \hspace{2.5cm}
+ \sigma^{\top}(t) a g_{\hat z \hat z} (t, \operatorname{Tr}(u\xi)) a^{\top} \sigma(t)
+ \sigma^{\top}(t) a g_{\hat z  z} (t, \operatorname{Tr}(u\xi)) \Sigma u
\\ \nonumber
& \hspace{2.5cm} 
+ u \Sigma^{\top} g^{\top}_{\hat z  z} (t, \operatorname{Tr}(u\xi)) a^{\top} \sigma(t)
+ \sigma^{\top}(t) a g_{\hat z \sqrt x} (t, \operatorname{Tr}(u\xi))
\biggr) m(d\xi),
\end{align}
for $(t,u) \in [0,T] \times S_d$.
The functions $\mathscr L(t)$ and $\mathscr C(t)$ are given by
\begin{align*}
  \mathscr L(t) &=
\frac12 c_y(t)
+ c_{z \sqrt{x}}^{\top}(t) \Sigma 
+ \sigma(t)^{\top} a c_{\hat z z}(t) \Sigma
\\
\mathscr C(t) &=
c_x(t)
+ \sigma(t)^{\top} a c_{\hat z \hat z}(t) a^{\top} \sigma(t)
+\sigma^{\top}(t) a c_{\hat z \sqrt x}(t)
+a o_2(t),
\end{align*}
for all $t \in [0,T]$.
Let $w(\cdot, u, v ): [0,T] \to \R$ be the solution of
\begin{align}
\label{eqn:odew}
-\frac{\partial w(t,u,v)}{\partial t} 
&= \varpi(t,\Gamma(t,u),w(t,u,v)), \quad w(T,u,v)=v,
\end{align}
with
\begin{align*}
  \varpi(t,u,v) &= 
c_y(t) v + 
c_t(t) 
+ \operatorname{Tr}(a o_1(t))+ \operatorname{Tr}(u b) 
\\
& \quad 
+ \int_{S_d^+ \setminus \{0\}} \left(\operatorname{Tr}(u \xi) 
+ g_y(t,\operatorname{Tr}(u\xi))v 
+ g_t(t,\operatorname{Tr}(u\xi)) \right) m(d\xi) ,
\end{align*}
for $(t,u,v) \in [0,T] \times S_d \times \R$.
Then the above BSDE has the unique solution
\begin{align}
\label{solmatrixBSDE}
\begin{cases}
  Y_t &=  \operatorname{Tr}(\Gamma(t,u)X_t) + \operatorname{Tr}(a O_t) + w(t,u,v),
\\ 
  Z_t&=2 \sqrt{X_t} \Gamma(t,u) \Sigma^{\top},
\\
  \hat Z_t &= \sqrt{X_t} \sigma^{\top}(t) a,
\\ 
  K_t(\xi)&= \operatorname{Tr}(\Gamma(t,u)\xi), 
\end{cases}
\end{align}
for all $\xi \in S_d^+$, $t \in [0,T]$.
\end{thm}

\begin{proof}
Let $\Gamma$ and $w$ be the unique solutions of \eqref{eqn:matrixode} and \eqref{eqn:odew}.
We apply It\^o's formula for semimartingales to the function $[0,T] \times S_d^+  \times M_d \ni (t,x, o) \mapsto \operatorname{Tr}(\Gamma(t,u) x) + \operatorname{Tr}(a o) + w(t,u,v)$.
Using representation \eqref{representationofx} and \eqref{eqn:hatN} this gives
\begin{align*}
& \operatorname{Tr}(\Gamma(t,u) X_t) + \operatorname{Tr}(a O_t) + w(t,u,v) 
\\ \nonumber
&=
 \operatorname{Tr}(\Gamma(T,u) X_T) +w(T,u,v) + \operatorname{Tr}(a O_T) - \int_t^T 2 \operatorname{Tr} \left(\Sigma \Gamma(s,u) \sqrt{X_s}dW_s \right) 
\\ \nonumber
& \quad 
- \int_t^T \left( \operatorname{Tr}(\Gamma(s,u) b) 
+ \operatorname{Tr}(\Gamma(s,u)B(X_s))
+ \operatorname{Tr}\left(\Gamma(s,u) \int_{S_d^+ \setminus \{0\} } \chi(\xi) m (d\xi) \right) \right) ds
\\ \nonumber
& \quad 
- \int_t^T \int_{S_d^+ \setminus \{0\} } \operatorname{Tr}(\Gamma(s,u)\chi(\xi)) (\mu^X(ds , d\xi) - \nu(ds,d\xi)) 
\\ \nonumber
& \quad
- \int_t^T \int_{S_d^+ \setminus \{0\} } \operatorname{Tr}(\Gamma(s,u)(\xi-\chi(\xi)) \mu^X(ds,d\xi)
\\ \nonumber
& \quad 
- \int_t^T \left(\operatorname{Tr}\left(\frac{\partial \Gamma(s,u)}{\partial s} X_s\right) 
+ \frac{\partial w(s,u,v)}{\partial s}\right) ds
\\ \nonumber
& \quad 
- \int_t^T \operatorname{Tr}\left( a \sigma (s) \sqrt{X_s} d \hat Q  
\right)
- \int_t^T  \operatorname{Tr} \left(  a (o_1(s)+ o_2(s) X_s) \right)ds
 ,
\end{align*}
where we have used basic properties of the trace to derive
\begin{align*}
\operatorname{Tr}\l(\Gamma(t,u) (\sqrt{X_t} dW_t \Sigma + \Sigma^\top dW_t^\top \sqrt{X_t}) \r)
&= 
2  \operatorname{Tr}\l(\Sigma \Gamma(t,u) \sqrt{X_t} dW_t  \r) .
\end{align*}
Because of the admissibility conditions \eqref{eqn:integrabilitym} and \eqref{eqn:integrabilityM} and the integrability condition \eqref{eqn:integrability1} on the measures $m$ and $M$, we may write the above equation in the following form
\begin{align}
\label{eqn:forward}
& \operatorname{Tr}(\Gamma(t,u) X_t)  + \operatorname{Tr}(a O_t) + w(t,u,v) 
\\ \nonumber
&= 
 \operatorname{Tr}(\Gamma(T,u) X_T) +w(T,u,v) + \operatorname{Tr}(a O_T) - \int_t^T 2 \operatorname{Tr} \left(\Sigma \Gamma(s,u) \sqrt{X_s}dW_s \right) 
\\ \nonumber
& \quad 
- \int_t^T \left( \operatorname{Tr}(\Gamma(s,u) b) 
+ \operatorname{Tr}(\Gamma(s,u)B(X_s))
+ \operatorname{Tr}\left(\Gamma(s,u) \int_{S_d^+ \setminus \{0\} } \chi(\xi) m (d\xi) \right) \right) ds
\\ \nonumber
& \quad 
- \int_t^T \int_{S_d^+ \setminus \{0\} } \operatorname{Tr}(\Gamma(s,u)\xi) (\mu^X(ds , d\xi) - \nu(ds,d\xi)) 
\\ \nonumber
& \quad
- \int_t^T \int_{S_d^+ \setminus \{0\} } \operatorname{Tr}(\Gamma(s,u)(\xi-\chi(\xi)) \nu(ds,d\xi)
\\ \nonumber
& \quad 
- \int_t^T \left(\operatorname{Tr}\left(\frac{\partial \Gamma(s,u)}{\partial s} X_s\right) 
+ \frac{\partial w(s,u,v)}{\partial s}\right) ds
\\ \nonumber
& \quad 
- \int_t^T \operatorname{Tr}\left( a \sigma (s) \sqrt{X_s} d \hat Q 
\right)
- \int_t^T  \operatorname{Tr} \left(  a (o_1(s)+ o_2(s) X_s) \right)ds
 .
\end{align}
Hence the BSDE \eqref{defBSDE}
is solved by \eqref{solmatrixBSDE} provided the finite variation parts of \eqref{eqn:forward} and the BSDE
coincide, i.e. if 
\begin{align*}
&f(t,X_t,Y_t,Z_t,\hat Z,K_t)
\\
&= 
  -\operatorname{Tr}(\Gamma(t,u) b) - \operatorname{Tr}(\Gamma(t,u)B(X_t)) - \operatorname{Tr}\left(\Gamma(t,u) \int_{S_d^+ \setminus \{0\} } \chi(\xi) m (d\xi) \right) 
\\
& \quad 
- \int_{S_d^+ \setminus \{0\} } \operatorname{Tr}\left(\Gamma(t,u)(\xi-\chi(\xi))\right) \left(m(d\xi) + M(X_t,d\xi)\right) 
\\
& \quad -\operatorname{Tr}\left(\frac{\partial \Gamma(t,u)}{\partial t} X_t\right) 
- \frac{\partial w(t,u,v)}{\partial t}
- \operatorname{Tr} \left(a o_1(t)\right) - \operatorname{Tr}\left( a  o_2(t) X_t \right) 
,
\end{align*}
for $t \in [0,T]$.
Using the special form \eqref{matrixgenerator} of the generator $f$ and formulae \eqref{solmatrixBSDE} we calculate 
\begin{align*}
& f(t,X_t,Y_t,Z_t,\hat Z_t,K_t) 
\\
&=  
4 \operatorname{Tr}( \sqrt{X_t} \Gamma(t,u)\Sigma^{\top} c_{zz}(t) \Sigma \Gamma(t,u) \sqrt{X_t})
+ 2 \operatorname{Tr}(  \sqrt{X_t} \Gamma(t,u)\Sigma^{\top} c_{z \sqrt{x}}(t) \sqrt{X_t})
+  \operatorname{Tr}( c_x(t) X_t) 
\\
& \quad  
+ c_y(t) \operatorname{Tr}(\Gamma(t,u) X_t) 
+ c_y(t) w(t,u,v) 
+c_t(t) 
+ \operatorname{Tr}(\sqrt{X_t} \sigma^{\top}(t) a  
	c_{\hat z \hat z}(t) 
	a^{\top} \sigma(t) \sqrt{X_t} )
\\ \nonumber
& \quad
+ \operatorname{Tr}(\sqrt{X_t} \sigma^{\top}(t) a 
	c_{\hat z z}(t) 
	2 \Sigma \Gamma(t,u) \sqrt{X_t})
+\operatorname{Tr}(\sqrt{X_t} \sigma^{\top}(t) a 
	c_{\hat z \sqrt{x}}(t) 
	\sqrt{X_t})
\\ \nonumber
& \quad
+ \int_{S_d^+ \setminus \{0\}} g_M(t,\operatorname{Tr}(\Gamma(t,u)\xi)) M(X_t,d\xi)
\\
& \quad + \int_{S_d^+ \setminus \{0\}} 
\biggl( 
\operatorname{Tr}(2X_t \Gamma(t,u) \Sigma^{\top}  g_{z\sqrt{x}}(t,\operatorname{Tr}(\Gamma(t,u)\xi))) 
+ \operatorname{Tr}(g_{x}(t,\operatorname{Tr}(\Gamma(t,u)\xi)) X_t) 
\\ \nonumber
&  \hspace{2.5cm}
+ \operatorname{Tr}(g_{y}(t,\operatorname{Tr}(\Gamma(t,u)\xi)) \Gamma(t,u) X_t)
+g_y(t,\operatorname{Tr}(\Gamma(t,u)\xi)) w(t,u,v)
\\ \nonumber
&  \hspace{2.5cm}
+g_t(t)
+ \operatorname{Tr}(\sqrt{X_t} \sigma^{\top}(t) a 
	g_{\hat z \hat z}(t,\operatorname{Tr}(\Gamma(t,u)\xi)) 
	a^{\top} \sigma(t) \sqrt{X_t})
\\
& \hspace{2.5cm} 
+ \operatorname{Tr}(\sqrt{X_t} \sigma^{\top}(t) a 
	g_{\hat z z}(t,\operatorname{Tr}(\Gamma(t,u)\xi)) 
	2 \Sigma \Gamma(t,u) \sqrt{X_t})
\\
& \hspace{2.5cm}
+\operatorname{Tr}(\sqrt{X_t} \sigma^{\top}(t) a 
	g_{\hat z \sqrt{x}}(t,\operatorname{Tr}(\Gamma(t,u)\xi)) 
	\sqrt{X_t})
\biggr) m(d\xi) 
\\
&=
- \operatorname{Tr}(\Gamma(t,u) b) 
- \operatorname{Tr}(\Gamma(t,u)B(X_t)) 
\\
& \quad
- \int_{S_d^+ \setminus \{0\} } \operatorname{Tr}\left(\Gamma(t,u)\chi(\xi)  \right) m (d\xi)
- \int_{S_d^+ \setminus \{0\} } \operatorname{Tr}\left(\Gamma(t,u)(\xi-\chi(\xi))\right) \left(m(d\xi) + M(X_t,d\xi)\right) 
\\
& \quad
- \operatorname{Tr}\left(\frac{\partial \Gamma(t,u)}{\partial t} X_t\right) 
-\frac{\partial w(t,u,v)}{\partial t} 
- \operatorname{Tr}(a o_1(t)) - \operatorname{Tr}(a o_2(t) X_t)),
\end{align*}
where the last equality is obtained from \eqref{eqn:matrixode} and \eqref{eqn:odew}, the definition of the adjoint operator $B^*$ and basic properties of the trace.
\end{proof}

\begin{rmk}
 Obviously our results also apply to 'standard' FBSDEs, where the BSDE is only allowed to depend the affine process $X$ itself. 
\end{rmk}

In the remaining section we study the existence of generalized Riccati equations of the form \eqref{eqn:matrixode}.
We always assume that $(\alpha,b,\beta^{ij},m,\mu)$ is an admissible parameter set and that \eqref{eqn:integrability1} holds.
Furthermore we make use of the fact that the ODE \eqref{eqn:matrixode} is equivalent to 
\begin{align}
\label{eqn:matrixodeforward}
  \frac{ \partial \Gamma(t,u)} {\partial t} = \theta(t, \Gamma(t,u)) , \quad \Gamma(0,u)=u,
\end{align}
for $(t,u) \in [0,T] \times S_d$.
This is derived via time reversal. 
The equations \eqref{eqn:matrixode} and \eqref{eqn:odew} are not coupled, in particular given $\Gamma$ we need only to solve a linear ODE to get $w$. 
Thus, our main objective is to find sufficient conditions for existence and uniqueness of the solution $\Gamma$ of \eqref{eqn:matrixode}.

\subsection{Solution of matrix Riccati ODEs in the continuous case}
\label{sec:solcont}

First we consider a simpler version of equation \eqref{eqn:matrixode}, more precisely
\begin{align}
\label{eqn:conttheta}
-\frac{ \partial \Gamma(t,u)} {\partial t} &= \theta(t, \Gamma(t,u)) , \quad \Gamma(T,u)=u,
\\ \nonumber
\theta(t,u)
&=
4 u \Sigma^{\top} c_{zz}(t) \Sigma u 
+ \mathscr L(t) u 
+ u \mathscr L^{\top}(t) 
+ B^*(u)
+ \mathscr C(t).
\end{align}
Note that for arbitrary $u \in S_d$ the solution of the above ODE may explode due to the quadratic term in the equation.
In the corresponding continuous case with constant coefficients and terminal value $u=0$ we derive a closed form solution.
We start with an assumption.
\begin{itemize}
  \item[\bf{(A1)}]
Let the linear drift term $B$ be of the form $B(x)= x\hat B +  \hat B^{\top}x$, where $\hat B \in M_d$.
\end{itemize}

\begin{prop} 
\label{easyriccati}
Let {\rm (A1)} hold and for $t \in [0,T]$ let the coefficients  $c_{zz}(t) \equiv c_{zz}, \mathscr L(t) \equiv \mathscr L, \mathscr C(t)  \equiv \mathscr C 
$ be constant.
Furthermore suppose $\alpha, c_{zz} \in S_d^{++}$ 
and define the matrix-valued functions $A:[0,T] \to M_{2d}$, $A_{ij}:[0,T] \to M_{d}$, $i,j=1,2$, in the following way
\begin{align*}
  A(t)
=\left( \begin{matrix}
             A_{11}(t)  & A_{12}(t) \\
  	     A_{21}(t)	& A_{22}(t)
             \end{matrix}
\right) 
=\exp \left( (T-t) \left( \begin{matrix}
		\mathscr L^{\top} + \hat B^{\top} & - 4\Sigma^{\top} c_{zz} \Sigma \\
	        \mathscr C & -\mathscr L - \hat B
			    c_x 	&	-(2 c_{z \sqrt x}^{\top} \Sigma + \hat B)
                          \end{matrix} \right)
\right).
\end{align*}
Then, for the terminal value $u =0$
there exists a unique solution $\Gamma(\cdot,0) \in S_d$ to \eqref{eqn:conttheta} given by
\begin{align*}
  \Gamma(t,0) = A^{-1}_{22}(t) A_{21}(t), \quad t \in [0,T].
\end{align*}
\end{prop}

\begin{proof}
The result relies on theorems from Lie group theory and can be found in \cite{gt08} Section 2.2.
We only outline the procedure in which they linearize the matrix Riccati ODE.
For every $t \in [0,T]$ let $J(t) \in M_d$ be an invertible matrix and $G(t) \in M_d$.
We set $J(T)=I_d$, $G(T)=0$ and 
\begin{align} 
\label{eqn:setgamma}
  \Gamma(t,0)&= J^{-1}(t)G(t), \quad  t \in [0,T],
\end{align}
which gives
\begin{align*}
  - \partial_t G(t) = -\partial_t (J(t) \Gamma(t,0)) = -\partial_t (J(t)) \Gamma(t,0) - J(t) \partial_t (\Gamma(t,0)),
\end{align*}
and hence using \eqref{eqn:conttheta} we have
\begin{align*}
\partial_t (J(t)) \Gamma(t,0) -\partial_t G(t)  
= &
 4 G(t) \Sigma^{\top} c_{zz} \Sigma \Gamma(t,0) 
+ G(t) \mathscr L^{\top} 
+ J(t)\mathscr L \Gamma(t,0)  
& + J(t) \hat B \Gamma(t,0) + G(t) \hat B^{\top} + J(t) \mathscr C. 
\end{align*}
Equating coefficients, we get a system of $2d$ linear equations
\begin{align}
\label{linearizedode}
 -\partial_t \left( \begin{matrix}
              G(t) & J(t)
            \end{matrix} \right)
& =
\left( \begin{matrix}
              G(t) & J(t)
       \end{matrix} \right)
\left( \begin{matrix}
		\mathscr L^{\top} + \hat B^{\top} & - 4\Sigma^{\top} c_{zz} \Sigma \\
	        \mathscr C & -\mathscr L - \hat B
       \end{matrix} \right),
\end{align}
with terminal value $(G(T),J(T))=(0,I_d)$.
Via exponentiation we deduce its solution to be
\begin{align*}
\left( \begin{matrix}
              G(t) & J(t)
            \end{matrix} \right)
& =
\left( \begin{matrix}
              0 & I_d
       \end{matrix} \right)
\exp \left( (T-t)
\left( \begin{matrix}
		\mathscr L^{\top} + \hat B^{\top} & - 4\Sigma^{\top} c_{zz} \Sigma \\
	        \mathscr C & -\mathscr L - \hat B
       \end{matrix} \right)\right).
\end{align*}
\end{proof}

\begin{rmk}
\label{rmk:integration}
  In the situation of Proposition \ref{easyriccati} we can further simplify the computation of the component $Y$ of the solution to BSDE \eqref{defBSDE} in Theorem \ref{thm:matrixBSDEsolution}. 
  In order to determine $Y$ it is typically not only necessary to solve \eqref{eqn:conttheta}, but also to numerically integrate $\operatorname{Tr}(\Gamma(t,u)b)$.
  This can be circumvented in many cases, compare \cite{gt08} Chapter 4.2.
  	Let $c_t(t) \equiv c_t$, $c_y=0$ and $o_1(t) \equiv o_1$ for all $t \in [0,T]$.
  With the notation from the above proof we have from \eqref{linearizedode} 
  \begin{align*}
	G(t)&= \left(\partial_t J(t) - J(t)\mathscr L  - J(t) \hat B\right) (4\Sigma^{\top}c_{zz} \Sigma)^{-1},
  \end{align*}
  and hence, using \eqref{eqn:setgamma} we get
  \begin{align*}
    \Gamma(t,0) &= \left(J^{-1}(t) \partial_t J(t) - \mathscr L - \hat B \right) (4\Sigma^{\top}c_{zz} \Sigma)^{-1}.
  \end{align*}
  Since
  \begin{align*}
    w(t,0,v)= v + \int_t^T \operatorname{Tr}(\Gamma(s,0)b) ds + \int_t^T \left( c_t + \operatorname{Tr}(a o_1)\right) ds,
  \end{align*}
  it follows with \eqref{solmatrixBSDE} for all $t \in [0,T]$
  \begin{align*}
    Y_t &=  \operatorname{Tr}(\Gamma(t,0) R_t) + \operatorname{Tr} (a O_t)+ v - \log \parallel 
    A_{22}(t) \parallel \operatorname{Tr}\left((4\Sigma^{\top}c_{zz} \Sigma)^{-1} b \right) \\
   & \quad - (T-t) \operatorname{Tr}\left((\mathscr L + \hat B) (4\Sigma^{\top}c_{zz} \Sigma)^{-1}b\right)  + (T-t) \left(c_t + \operatorname{Tr}(a o_1) \right) .
  \end{align*}
\end{rmk}

To state further existence results we need additional assumptions, namely
\begin{itemize}
  \item[\bf{(A2$^+$)}] $c_{zz} \in S_d^-$ and $\mathscr C \in S_d^+$,
  \item[\bf{(A2$^-$)}] $c_{zz} \in S_d^+$ and $\mathscr C \in S_d^-$.
\end{itemize}

\begin{prop}
\label{christaeberhardriccati}
Let {\rm (A2$^+$)} hold.
Then, for every $u \in S_d^+$, there exists a unique solution $\Gamma(\cdot,u) \in S_d^+$ to \eqref{eqn:conttheta}.
Moreover, if $t \mapsto (c_{zz}(t), \mathscr L(t),\mathscr C(t) )$ is real analytic in $t \in [0,T]$ and $u \in S_d^{++}$, then $\Gamma(\cdot , u) \in S_d^{++}$.
\end{prop}

\begin{proof}
For $u \in S_d^+$ the existence of a solution $\Gamma(\cdot , u ) \in S_d^+$ of \eqref{eqn:matrixode} follows from \cite{cm08} Theorem 4.3.
Indeed the continuous parameter set $t \mapsto (-4 \Sigma^{\top} c_{zz}(t) \Sigma, l(t), \mathscr C(t))$ with $l(t)(u)= u\mathscr L^{\top}(t)+\mathscr L(t) u +B^*(u)$, $u \in S_d$, is admissible in the sense: $-4 \Sigma^{\top} c_{zz}(t) \Sigma, \mathscr C(t) \in S_d^+$ and for all $u_0 \in \partial S_d^+$, $w \in S_d^+$ such that $\operatorname{Tr}(u_0 w ) =0$ and $t \in [0,T]$ we have
\begin{align*}
\operatorname{Tr}(l(t)(u_0)w) &= \operatorname{Tr}(u_0 \mathscr L^{\top}(t) w+\mathscr L(t) u_0 w+B^*(u_0)w)\\
&\geq   \operatorname{Tr}(u_0 \mathscr L^{\top}(t) w +  \mathscr L(t) u_0 w)
\\
 & = 0, 
\end{align*}
where we have used $\operatorname{Tr}(B^*(x)y)=\operatorname{Tr}(B(y)x)$ for all $x,y\in S_d$, as well as admissibility condition \eqref{eqn:lindrift} and \cite{cfmt09} Lemma 4.1.
Since the right hand side of \eqref{eqn:conttheta} is locally Lipschitz in $u$, uniqueness follows immediately.
\cite{cm08} Theorem 4.3 now implies the second assertion.
\end{proof}

\begin{cor}
\label{cor:christaeberhardriccati}
If {\rm (A2$^-$)} holds,
then there exists a unique solution $\Gamma(\cdot, u) \in S_d^-$ to \eqref{eqn:conttheta} for every $u \in S_d^-$.
Moreover, if $t \mapsto (c_{zz}(t), \mathscr L(t),\mathscr C(t) )$ is real analytic in $t \in [0,T]$ and $u \in S_d^{--}$, then $\Gamma(\cdot, u) \in S_d^{--}$.
\end{cor}

\begin{proof}
Fix $u \in S_d^-$.
Finding a solution $\Gamma(\cdot, u)$ to the ODE \eqref{eqn:conttheta} is equivalent to solving
\begin{align}
\label{eqn:odenegdef}
  -\frac{\partial \tilde \Gamma(t,-u)}{\partial t} & = \tilde \theta (t,\tilde \Gamma(t,-u)), \quad \tilde \Gamma(T,-u)=-u,
\end{align}
where $\tilde \Gamma (t,-u)= -\Gamma(t,u)$ and
\begin{align*}
  \tilde \theta(t,\tilde u) & = -\theta(t,- \tilde u) 
=-4 \tilde u \Sigma^{\top} c_{zz}(t) \Sigma \tilde u 
+ \mathscr L(t) \tilde u 
+ \tilde u \mathscr L^{\top}(t) 
+ B^*(\tilde u)
- \mathscr C(t)
\end{align*}
for $(t,\tilde u) \in [0,T] \times S_d^+$. 
Assumption {\rm (A2$^+$)} and Proposition \ref{christaeberhardriccati} imply that the ODE \eqref{eqn:odenegdef} has a unique solution $\tilde \Gamma(\cdot, -u) \in S_d^+$.
Hence the unique solution of \eqref{eqn:conttheta} is given by $\Gamma(\cdot ,u)= - \tilde \Gamma (\cdot,-u) \in S_d^-$.
The second assertion follows immediately.
\end{proof}

\subsection{Solution of generalized matrix Riccati ODEs}

Before we derive sufficient conditions for the existence and uniqueness 
of equation \eqref{eqn:matrixode} in the presence of jumps, we provide some properties of the function $\theta$ given in \eqref{eqn:matrixodetheta}.
This enables us to use a comparison result for ODEs and prove existence of \eqref{eqn:matrixode} with the help of the above existence results.
We introduce quasi-monotone increasing functions in the sense of \cite{v73}.

\begin{defn}
Let $U \subset S_d$ be a an open set.
We call a function $\theta: U \to S_d$ quasi-monotone increasing if for all $w,u \in U$ and $x \in S_d^+$ with $\operatorname{Tr}(wx)=0$ the inequality
\begin{align*}
  \operatorname{Tr}((\theta(w+u)-\theta(u)) x) \geq 0
\end{align*}
holds.
Correspondingly we call $\theta$ {\rm quasi-constant} if both $\theta$ and $-\theta$ are quasi-monotone increasing.
\end{defn}

We work under the following conditions which ensure that $\theta$, defined in \eqref{eqn:matrixodetheta} is quasi-monotone increasing.

\begin{itemize}
  \item[\bf{(A3$^+$)}] For all $t \in [0,T]$, the functions $g_M(t,\cdot)$, $g_x(t,\cdot)$, $g_{z \sqrt x} (t, \cdot)$ and $g_y(t, \cdot)$ are non-decreasing on $\R_+$.
For all $(t,r) \in [0,T] \times \R_+$, we have $g_M(t, r) \in \R_+$, $g_x(t,r) \in S_d^+$ and 
\begin{align*}
\Sigma^{\top} g_{z \sqrt x}(t,r)  
	+ g^{\top}_{z \sqrt x}(t,r)  \Sigma
+g_y(t,r)
& \succeq 0.
\end{align*}
  \item[\bf{(A3$^-$)}] For all $t \in [0,T]$, the functions $g_M(t,\cdot)$, $g_x(t,\cdot)$ are non-decreasing on $\R_-$ and the functions $g_{z \sqrt x} (t, \cdot)$, $g_y(t,\cdot)$ are non-increasing on $\R_-$.
For all $(t,r) \in [0,T] \times \R_-$, we have $g_M(t, r) \in \R_-$, $g_x(t,r) \in S_d^-$ and 
\begin{align*}
\Sigma^{\top}  g_{z \sqrt x}(t,r) 
	+ g^{\top}_{z \sqrt x}(t,r)  \Sigma
+g_y(t,r)
& \succeq 0.
\end{align*}
  \item[\bf{(A4$^+$)}] For all $t \in [0,T]$, the functions 
$g_{\hat z \hat z}(t,\cdot)$, 
$g_{\hat z z }(t ,\cdot)$, 
$g_{\hat z \sqrt x }(t ,\cdot)$
are non-decreasing on $\R_+$.
  For all $(t,y) \in [0,T] \times \R_+$ we have 
$g_{\hat z \hat z}(t,y) \in S_d^+$ and
\begin{align*} 
 \sigma^{\top}(t) a 
	 g_{\hat z  z} (t, y)
	\Sigma 
+ \Sigma^{\top}  g_{\hat z  z} (t,y)
a^{\top} \sigma(t)
&  \succeq 0,
\\ \nonumber
 \sigma^{\top}(t) a 
	 g_{\hat z \sqrt x} (t,y)
& \succeq 0.
\end{align*}
  \item[\bf{(A4$^-$)}] For all $t \in [0,T]$, the functions 
$g_{\hat z \hat z}(t,\cdot)$, 
$g_{\hat z \sqrt x }(t ,\cdot)$ 
are non-decreasing on $\R_-$ and 
$g_{\hat z z}(t,\cdot)$ 
are non-increasing on $\R_-$.
  For all $(t,y) \in [0,T]\times \R_-$ we have 
\begin{align*} 
 \sigma^{\top}(t) a 
	 g_{\hat z  z} (t, y)
	\Sigma 
+ \Sigma^{\top} g_{\hat z  z} (t, y)
a^{\top} \sigma(t)
&  \succeq 0,
\\ \nonumber
 \sigma^{\top}(t) a 
	 g_{\hat z \sqrt x} (t,y) 
& \preceq 0.
\end{align*}
\end{itemize}

\begin{lem}
\label{lem:qminc}
Suppose {\rm (A3$^+$)} and {\rm (A4$^+$)} hold.
Then the map $\theta(t,\cdot)$ 
is quasi-monotone increasing on $S_d^{++}$ for all $t \in [0,T]$.
\end{lem}

\begin{proof}
Consider an arbitrary $t \in [0,T]$ and fix $w,v \in S_d^{++}$ and $r \in S_d^+$ such that $\operatorname{Tr}(wr)=0$.
First we consider the function
\begin{align*}
  u \mapsto 4 u \Sigma^{\top} c_{zz}(t) \Sigma u, \quad u \in S_d^{++},
\end{align*}
and prove it is quasi-constant.
Indeed, we have
\begin{align*}
&  \operatorname{Tr}\left(\left((w+v)4\Sigma^{\top} c_{zz}(t) \Sigma(w+v)-v 4\Sigma^{\top} c_{zz}(t) \Sigma v \right)r\right)
\\
&
= \operatorname{Tr}\left(w4\Sigma^{\top} c_{zz}(t) \Sigma w r\right) + \operatorname{Tr}\left(v 4\Sigma^{\top} c_{zz}(t) \Sigma wr\right) + \operatorname{Tr}\left(w4\Sigma^{\top} c_{zz}(t) \Sigma vr\right)
\\
&=0,
\end{align*}
where the last equality follows from $\operatorname{Tr}(wr)=0$ being equivalent to $rw=wr=0$, compare \cite{cfmt09} Lemma 4.1.
With the same reasoning we have
\begin{align*}
  u \mapsto u \mathscr L^{\top}(t) +  u \mathscr L(t) + \mathscr C(t), \quad u \in S_d^{++},
\end{align*}
is quasi-constant.
Moreover due to the admissibility condition \eqref{eqn:lindrift} the function
\begin{align*}
  u \mapsto B^*(u) - \int_{S_d^+ \setminus \{0\}} 
\frac{\operatorname{Tr}(u\chi(\xi))}{\parallel \xi \parallel^2 \wedge 1} \mu(d\xi)  , \quad u \in S_d^{++},
\end{align*}
is quasi-monotone increasing.
From the definition of $\mu$ together with the fact that $g_M$ is non-decreasing in its spatial variable and $\operatorname{Tr}(xy)\geq 0$ for all $x,y \in S_d^+$, the map
\begin{align*}
  u \mapsto 
\int_{S_d^+ \setminus \{0\}} 
\frac{g_M(t,\operatorname{Tr}(u\xi)) + \operatorname{Tr}(u\xi)}{\parallel \xi \parallel^2 \wedge 1} \mu(d\xi), \quad u \in S_d^{++},
\end{align*}
is quasi-monotone increasing.
Due to {\rm (A3$^+$)} we have
\begin{align*}
&\operatorname{Tr} \Biggl( r
\int_{S_d^+ \setminus \{0\}} \Bigl(
(w+v) \Sigma^{\top} g_{z \sqrt{x}}(t,\operatorname{Tr}((w+v)\xi))
+ g_{z \sqrt{x}}^{\top}(t,\operatorname{Tr}((w+v)\xi)) \Sigma (w+v) 
\\
&  \hspace{2.6cm}
-v \Sigma^{\top} g_{z \sqrt{x}}(\operatorname{Tr}(t,\operatorname{Tr}(v\xi))) 
- g_{z \sqrt{x}}^{\top}(t,\operatorname{Tr}(v\xi)) \Sigma v 
+ (w+v) g_y(t,\operatorname{Tr}((w+v)\xi))
\\
& \hspace{2.6cm}
- v g_y(t,\operatorname{Tr}(v\xi))
+ g_x(t,\operatorname{Tr}((w+v)\xi)) 
- g_x(t,\operatorname{Tr}(v\xi)) 
\Bigr) m(d\xi)
\Biggr)
\\
&=
\operatorname{Tr} \Biggl( r
\int_{S_d^+ \setminus \{0\}} \Bigl(
v \Sigma^{\top} g_{z \sqrt{x}}(t,\operatorname{Tr}((w+v)\xi))
+ g_{z \sqrt{x}}^{\top}(t,\operatorname{Tr}((w+v)\xi)) \Sigma v
\\
& \hspace{3cm}
-v \Sigma^{\top} g_{z \sqrt{x}}(\operatorname{Tr}(t,\operatorname{Tr}(v\xi))) 
- g_{z \sqrt{x}}^{\top}(t,\operatorname{Tr}(v\xi)) \Sigma v 
+ v g_y(t,\operatorname{Tr}((w+v)\xi)) 
\\
& \hspace{3cm}
- v g_y(t,\operatorname{Tr}(v\xi)) 
+ g_x(t,\operatorname{Tr}((w+v)\xi)) 
- g_x(t,\operatorname{Tr}(v\xi)) 
\Bigr) m(d\xi)
\Biggr)
\\
& \geq 0
,
\end{align*}
which implies that
\begin{align*}
u \mapsto &
\int_{S_d^+ \setminus \{0\}} \Bigl(
 u \Sigma^{\top} g_{z \sqrt{x}}(t,\operatorname{Tr}(u\xi))
+ g_{z \sqrt{x}}^{\top}(t,\operatorname{Tr}(u\xi)) \Sigma u 
+ g_{y}((t,\operatorname{Tr}(u\xi)) 
+ g_x(t,\operatorname{Tr}(u\xi)) 
\Bigr) m(d\xi)
,
\end{align*}
is a quasi-monotone increasing function for all $u \in S_d^{++}$.
We finally obtain from (A4$^+$) that the function
\begin{align*}
&u \mapsto \int_{ S_d^+ \setminus \{0\}} 
\left(
\sigma^{\top}(t) a g_{\hat z \hat z} (t, \operatorname{Tr}(u\xi)) a^{\top} \sigma(t)
+ \sigma^{\top}(t) a g_{\hat z  z} (t, \operatorname{Tr}(u\xi)) \Sigma u
 \right.
 \\ \nonumber
 &  \hspace{2.5cm}  \left.
+ u \Sigma^{\top} g^{\top}_{\hat z  z} (t, \operatorname{Tr}(u\xi)) a^{\top} \sigma(t)
+ \sigma^{\top}(t) a g_{\hat z \sqrt x} (t, \operatorname{Tr}(u\xi))
\right) m(d\xi),
\end{align*}
is quasi-monotone increasing for all $u \in S_d^{++}$.
\end{proof}

In the following Lemma we show that the growth of $\theta(t,u)$ in $u$, where $(t,u) \in [0,T] \times S_d^+$, can be controlled under suitable conditions.

\begin{itemize}
  \item[\bf{(A5$^+$)}] 
The function $c_{zz}$ has values in $S_d^-$ and $g_M$ and $g_x$ are 
of linear growth in the spatial variable, i.e. there exist continuous functions $C_M:[0,T]\to \R_{++}$ and $C_x:[0,T] \to S_d^{++}$ such that
\begin{align*}
g_M(t,y) & \leq C_M(t) ( |y| + 1),
\\ \nonumber
 g_x(t,y) & \preceq C_x(t) ( |y| + 1),
\end{align*}
for all $(t,y) \in [0,T]\times \R_+$.
Also $g_{z \sqrt x}$ and $g_y$ are bounded above in the spatial variable, i.e. there exists a matrix valued continuous function $C_{z \sqrt x}:[0,T]\to S_d^{++}$ such that
\begin{align*}
g_{z \sqrt x}(t,r) + g_y(t,r) \preceq C_{z \sqrt x}(t) , \quad \mbox{ for all } (t,r) \in [0,T]\times \R_+.
\end{align*}
  \item[\bf{(A5$^-$)}] 
The function $c_{zz}$ has values in $S_d^+$ and $-g_M$ and $-g_x$ are 
of linear growth in the spatial variable, i.e. there exist continuous functions $C_M:[0,T]\to \R_{++}$ and $C_x:[0,T] \to S_d^{++}$ such that
\begin{align*}
-g_M(t,y) & \leq C_M(t) ( |y| + 1),
\\ \nonumber
- g_x(t,y) & \preceq C_x(t) ( |y| + 1),
\end{align*}
for all $(t,y) \in [0,T]\times \R_-$.
Also $g_{z \sqrt x}$ and $g_y$ are bounded above in the spatial variable, i.e. there exists a matrix valued continuous function $C_{z \sqrt x}:[0,T]\to S_d^{++}$ such that
\begin{align*}
 g_{z \sqrt x}(t,r) +g_y(t,r)\preceq C_{z \sqrt x}(t) , \quad \mbox{ for all } (t,r) \in [0,T]\times \R_-.
\end{align*}
  \item[\bf{(A6$^+$)}]  The functions  
$g_{\hat z \hat z}$ and
$g_{\hat z \sqrt x}$, 
are of linear growth in the spatial variable, i.e. there exist continuous functions 
$C_{\hat z \hat z}:[0,T] \to S_d^{++}$ such that
\begin{align*}
 g_{\hat z \hat z}(t,y) +  g_{\hat z \sqrt x}(t,y)& \preceq C_{\hat z \hat z}(t) (|y|+1), 
\end{align*}
for all $(t,y) \in [0,T]\times \R_+ $.
Also 
$g_{\hat z z}$ is bounded above in the spatial variable, i.e. there exists a continuous function $C_{\hat z z }:[0,T] \to S_d^{++}$ such that
\begin{align*}
g_{\hat z z}(t,y) & \preceq C_{\hat z z}(t) ,\quad \mbox{ for all } (t,y) \in [0,T]\times \R_+.
\end{align*}
 \item[\bf{(A6$^-$)}]  The functions 
$-g_{\hat z \hat z}$ and
$-g_{\hat z \sqrt x}$, 
are of linear growth in the spatial variable, i.e. there exists a continuous function 
$C_{\hat z \hat z}:[0,T] \to S_d^{++}$ such that
\begin{align*}
 -g_{\hat z \hat z}(t,y) -  g_{\hat z \sqrt x}(t,y)& \preceq C_{\hat z \hat z}(t) (|y|+1), 
\end{align*}
for all $(t,y) \in [0,T]\times \R_- $.
Also 
$g_{\hat z z}$ is bounded above in the spatial variable, i.e. there exists a continuous function
$C_{\hat z z }:[0,T] \to S_d^{++}$ such that
\begin{align*}
g_{\hat z z}(t,y) & \preceq C_{\hat z z}(t) ,\quad \mbox{ for all } (t,y) \in [0,T]\times \R_-.
\end{align*}
\end{itemize}

\begin{lem}
\label{lem:u2+1}
Let {\rm (A5$^+$)} and {\rm (A6$^+$)} hold.
Then there exists a continuous function $K:[0,T]\to \R_{++}$ such that for $(t,u) \in [0,T] \times S_d^+$ the following inequality holds
\begin{align*}
  \operatorname{Tr}(u \theta(t,u)) \leq K(t) \left( \parallel u \parallel^2 +1\right).
\end{align*}
\end{lem}

\begin{proof}
For every $(t,u) \in [0,T] \times S_d^{+}$ we deduce that
\begin{align*}
 \theta(t,u) 
& \preceq
  u \mathscr L^{\top}(t) + \mathscr L u + B^*(u) + \mathscr C(t)
\\
& \quad + \int_{S_d^+ \setminus \{0\} } 
\frac{C_M(t) \left(\operatorname{Tr}\left(u\xi \right) +1 \right)  + \operatorname{Tr}\left(u(\xi-\chi(\xi))\right)}{\parallel \xi \parallel^2 \wedge 1} \mu(d\xi) 
\\
& \quad 
+ \int_{S_d^+ \setminus \{0\}} \Bigl(
 u \Sigma^{\top} C_{z \sqrt x}(t) 
+ C^{\top}_{z \sqrt x}(t) \Sigma u 
+ C_x(t) (\operatorname{Tr}\left(u\xi \right) +1) 
\\ \nonumber
& \hspace{2.5cm} 
+ \sigma^{\top}(t) a C_{\hat z \hat z} (\operatorname{Tr}(u\xi) + 1) a^{\top} \sigma(t)
+  \sigma^{\top}(t) a C_{\hat z z}(t) \Sigma u
\\ \nonumber
& \hspace{2.5cm} 
+ \sigma^{\top}(t) a C_{\hat z \hat z}(t) (\operatorname{Tr}\left(u\xi \right) +1)
\Bigr) m(d\xi)
,
\end{align*}
where we have used the Assumptions (A5$^+$) and (A6$^+$). 
In particular for $u \in S_d^+$ it then follows that
\begin{align*}
 \operatorname{Tr}(u \theta(t,u)) 
& \leq
 \operatorname{Tr}(uu \mathscr L^{\top}(t)) 
+ \operatorname{Tr}( u\mathscr L(t) u) + \operatorname{Tr}(uB^*(u)) 
+ \operatorname{Tr}(u \mathscr C(t)) 
\\
& \quad 
+ \operatorname{Tr}\left( u\int_{S_d^+ \setminus \{0\} } 
\frac{ C_M(t) \left(\operatorname{Tr}\left(u\xi \right) +1 \right)+  \operatorname{Tr}\left(u(\xi-\chi(\xi))\right)}{\parallel \xi \parallel^2 \wedge 1} \mu(d\xi) \right)
\\
& \quad 
+ \int_{S_d^+ \setminus \{0\}} 
\operatorname{Tr} \biggl( u \Bigl(
u \Sigma^{\top} C_{z \sqrt x}(t) 
+ C^{\top}_{z \sqrt x}(t) \Sigma u 
+ C_x(t) (\operatorname{Tr}\left(u\xi \right) +1) 
\\ \nonumber
& \hspace{3cm}  
+ \sigma^{\top}(t) a C_{\hat z \hat z} (\operatorname{Tr}\left(u\xi \right) +1) a^{\top} \sigma(t)
+ \sigma^{\top}(t) a C_{\hat z z}(t) \Sigma u
 \\ \nonumber
 & \hspace{3cm}  
+ \sigma^{\top}(t) a C_{\hat z \hat z}(t) (\operatorname{Tr}\left(u\xi \right) +1)
\Bigr) \biggr) m(d\xi)
,
\end{align*}
which gives the result.
\end{proof}

The following theorem goes back to \cite{v73}.
In the form stated here, it is adjusted to symmetric matrices. 
\begin{thm}
\label{thm:volkmann}
  Let $U \subset S_d$ be an open set.
Suppose $\theta: [0,T) \times U \to S_d$ is a continuous locally Lipschitz map with $\theta(t, \cdot)$ quasi-monotone increasing on $U$ for all $t \in [0,T)$.
Given $t_0 \in (0, T]$ let $x,y: [0,t_0) \to U$ be differentiable maps such that $x(0) \preceq y(0)$ and 
\begin{align*}
 \frac{dx(t)}{dt}-\theta(t,x(t)) \preceq \frac{dy(t)}{dt}-\theta(t,y(t)), \quad t \in [0,t_0].
\end{align*}
Then we have $x(t) \preceq y(t)$ for all $t \in [0,t_0)$.
\end{thm}
We will use this theorem for proving an existence and uniqueness result for the generalized Riccati ODE \eqref{eqn:matrixode}.

\begin{itemize}
  \item[\bf(A7)] The mappings $g_M, g_{ z \sqrt x},g_x,g_y,
g_{\hat z \hat z},
g_{\hat z z},$
and $g_{\hat z \sqrt x}$ are locally Lipschitz continuous in the spatial variable.

\end{itemize}
Note that the map $\theta(t,\cdot)$ may not be Lipschitz at the boundary $\partial S_d^+$, see \cite{dfs03} Example 9.3.
So we look for a solution of \eqref{eqn:matrixode} in $S_d^{++}$.

\begin{thm}
\label{prop:existencematrixode}
Let the map $t \mapsto (c_{zz}(t),\mathscr L(t),\mathscr C(t))$ be real analytic and suppose {\rm (A2$^+$)}, {\rm (A3$^+$)}, {\rm (A4$^+$)}, {\rm (A5$^+$)}, {\rm (A6$^+$)} 
and {\rm (A7)} hold.
Then the ODE \eqref{eqn:matrixode} has a unique solution $\Gamma(\cdot,u) \in S_d^{++}$ for every terminal value $u \in S_d^{++}$.
\end{thm}

\begin{proof}
Since $u \mapsto \theta(t,u)$ is locally Lipschitz on $S_d^{++}$ and $t \mapsto \theta(t,u)$ is continuous, standard ODE theory gives that there exists a unique local $S_d^{++}$-valued solution $\Gamma(t,u)$ of \eqref{eqn:matrixodeforward} for $t \in [0,t_+(u))$ with
\begin{align*}
  t_+(u)= \liminf_{n \to \infty} \left\{ t \geq 0 \ : \ \parallel \Gamma(t,u) \parallel \geq n \mbox{ or } \Gamma(t,u) \in \partial S_d^+ \right\} \wedge T.
\end{align*}
Wanting a global solution we thus need to show that $t_+(u)=T$.
Since $\theta(t, \cdot)$ may fail to be Lipschitz continuous at $\partial S_d^+$, we look at $\theta$ without its jump terms and define 
\begin{align*}
  \tilde \theta(t,u) &=
4 u \Sigma^{\top} c_{zz}(t) \Sigma u + u \mathscr L^{\top}(t) + \mathscr L(t) u + B^*(u) 
+ \mathscr C(t), 
\end{align*}
for $(t,u) \in [0,T] \times S_d$.
By assumption (A2$^+$) and since the map $t \mapsto (c_{zz}(t), \mathscr L(t), \mathscr C(t))$ is real analytic we may apply Proposition \ref{christaeberhardriccati}.
Hence there exists a unique $S_d^{++}$-valued solution $\tilde \Gamma$ of
\begin{align*}
  \frac{\partial \tilde \Gamma(t,u)}{\partial t} = \tilde \theta (t,\tilde \Gamma(t,u)), \quad \tilde \theta(0,u)=u,
\end{align*}
for all $t \in [0,T]$.
By means of (A3$^+$) and (A4$^+$) we have that for all $(t,u) \in [0,T] \times S_d^+$
\begin{align*}
&  \theta(t,u) - \tilde \theta (t,u) 
\\
&=  \int_{S_d^+ \setminus \{0\}} \frac{g_M(t,\operatorname{Tr}(u\xi)) + \operatorname{Tr}(u(\xi-\chi(\xi)))}{\parallel \xi \parallel^2 \wedge 1} \mu(d\xi) 
\\
& \quad
+ \int_{S_d^+ \setminus \{0\}} \left(
2u \Sigma^T g_{z \sqrt{x}}(\operatorname{Tr}(t,\operatorname{Tr}(u\xi))) 
+ g_x(t,\operatorname{Tr}(u\xi)) 
+ g_y(t,\operatorname{Tr}(u\xi))
\right) m(d\xi)
\\ \nonumber
& \quad 
+ \int_{S_d^+ \setminus \{0\}} \left(
\sigma^{\top}(t) a g_{\hat z \hat z} (t, \operatorname{Tr}(u\xi)) a^{\top} \sigma(t)
+ 2 \sigma^{\top}(t) a g_{\hat z  z} (t, \operatorname{Tr}(u\xi)) \Sigma u
\right.
\\ \nonumber
& \hspace{2.5cm}  \left.
+ \sigma^{\top}(t) a g_{\hat z \sqrt x} (t, \operatorname{Tr}(u\xi))
\right) m(d\xi)
\\
& \succeq 0,
\end{align*}
and thus for all $u \in S_d^{++}$ and $t \in [0,t_+(u))$ we obtain
\begin{align*}
\frac{\partial \Gamma (t,u)}{\partial t} - \tilde \theta (t, \Gamma (t,u)) 
\succeq
\frac{\partial \tilde \Gamma (t,u)}{\partial t} - \tilde \theta (t,\tilde \Gamma (t,u))  .
\end{align*}
Theorem \ref{thm:volkmann} then yields for $t \in [0,t_+(u))$
\begin{align*}
  \Gamma(t,u) \succeq \tilde \Gamma(t,u) \in S_d^{++},
\end{align*}
and hence
$t_+(u)=\liminf_{n \to \infty} \{t \geq 0 \ : \ \parallel \Gamma(t,u) \parallel \geq n \}\wedge T$.
By assumptions (A5$^+$) and (A6$^+$) we have from Lemma \ref{lem:u2+1} that for all $u \in S_d^{++}$ and $t \in [0,t_+(u))$
\begin{align*}
  \partial_t \parallel \Gamma (t,u) \parallel^2 &=
2 \operatorname{Tr}\left( \Gamma (t,u) \partial_t  \Gamma (t,u)\right) \leq  K(t) \left( \parallel \Gamma (t,u) \parallel^2 +1 \right),
\end{align*}
for a continuous positive-valued function $K$.
Then Gronwall's inequality applied to $\parallel \Gamma (t,u) \parallel^2+1$ gives
\begin{align*}
\parallel \Gamma (t,u) \parallel^2 & \leq 
e^{\int_0^t K(s)ds} \left( \parallel u \parallel^2 + 1 \right) \mbox{ for } t < t_+(u),
\end{align*}
and thus $t_+(u)=T$ for $u \in S_d^{++}$.
\end{proof}

\begin{cor}
\label{cor:sd-mitsprung}
Let the map $t \mapsto (c_{zz}(t),\mathscr L(t),\mathscr C(t))$ be real analytic and suppose {\rm(A2$^-$)}, {\rm(A3$^-$)}, {\rm(A4$^-$)}, {\rm(A5$^-$)}, {\rm(A6$^-$)} and {\rm (A7)} hold.
Then the ODE \eqref{eqn:matrixode} has a unique solution $\Gamma(\cdot,u) \in S_d^{--}$ for every terminal value $u \in S_d^{--}$.
\end{cor}

\begin{proof}
Fix $u \in S_d^{--}$. 
As we have already seen in the proof of Corollary \ref{christaeberhardriccati}, finding a solution $\Gamma(\cdot, u)$ to the ODE \eqref{eqn:matrixode} is equivalent to solving
\begin{align}
\label{eqn:odenegdef2}
  -\frac{\partial \tilde \Gamma(t,-u)}{\partial t} & = \tilde \theta (t,\tilde \Gamma(t,-u)), \quad \tilde \Gamma(T,-u)=-u,
\end{align}
where $\tilde \Gamma (t,-u)= -\Gamma(t,u)$ and
\begin{align*}
  \tilde \theta(t,\tilde u) 
& = -\theta(t,- \tilde u) 
\\
&=
- 4 \tilde u \Sigma^{\top} c_{zz}(t) \Sigma \tilde u 
+ \mathscr L(t) \tilde u 
+ \tilde u \mathscr L^{\top}(t) 
+ B^*(\tilde u)
- \mathscr C(t)
\\ \nonumber
& \quad 
+ \int_{S_d^+ \setminus \{0\}} 
\frac{\operatorname{Tr}\left(\tilde u(\xi-\chi(\xi))\right) - g_M(t,- \operatorname{Tr}(\tilde u\xi)) }{\parallel \xi \parallel^2 \wedge 1} \mu(d\xi)
\\ \nonumber
& \quad 
+ \int_{S_d^+ \setminus \{0\}} \Bigl(
g_{z \sqrt{x}}^{\top}(t,-\operatorname{Tr}(\tilde u\xi)) \Sigma u 
- g_x(t,- \operatorname{Tr}(\tilde u\xi)) 
 \\ \nonumber
 & \hspace{2.5cm}
+  \sigma^{\top}(t) a g_{\hat z  z} (t, - \operatorname{Tr}(\tilde u\xi)) \Sigma \tilde u
- \sigma^{\top}(t) a g_{\hat z \sqrt x} (t,- \operatorname{Tr}(\tilde u\xi))
\Bigr) m(d\xi)
\end{align*}
for $(t,\tilde u) \in [0,T] \times S_d^+$. By Theorem \ref{prop:existencematrixode} the ODE \eqref{eqn:odenegdef2} has a unique solution $\tilde \Gamma(\cdot, -u) \in S_d^+$ and hence there exists a unique solution $\Gamma(\cdot ,u)= - \tilde \Gamma (\cdot,-u) \in S_d^-$.
\end{proof}

\section{Application in multivariate affine stochastic volatility models}
\label{sec:stochvol}

In this chapter we apply the results of the previous chapter to the classical problem of utility maximization in a multivariate stochastic volatility setting.
Stochastic volatility models are an extension of the Black-Scholes model, where the previously constant assumed volatility is now modeled as a stochastic process.
The key feature of affine stochastic volatility models is that their Fourier-Laplace transform has an exponentially affine form.
For a multivariate model consider the $d$-dimensional logarithmic price process $N$ whose stochastic volatility is given by an affine process $R$ on $S_d^+$, then the following formula holds
\begin{align*}
  \E \l[ e^{\operatorname{Tr}(uR_t) + v^\top N_t}\r] &= e^{\operatorname{Tr}(\Psi(t,u,v)R_0)+v^\top N_0 + \Phi(t,u,v)},
\end{align*}
for suitable arguments $t \in [0,T]$, $u \in S_d +$i$S_d$ and $v \in \mathbb C^d$. 
The functions $\Phi$ and $\Psi$ solve a system of generalized Riccati ODEs which are specified by the model parameters.
This formula is the main reason for the analytic tractability of affine stochastic volatility models.
In the multivariate stochastic volatility models mainly used in the literature, the dynamics of $R$ follow
\begin{align*} 
  dR_t &= (b+ \hat B R_t + R_t \hat B^\top) dt + \sqrt{R_t}dW_t \Sigma + \Sigma^\top dW^\top \sqrt{R_t} + dJ_t,
\\
R_0&=r \in S_d^+,
\end{align*}
where $W$ is a matrix-valued Brownian motion possibly correlated with the Brownian motion driving $N$.
Moreover $b$ is a suitably chosen matrix in $S_d^+$, $\Sigma$, $\hat B$ are some invertible matrices and $J$ is a pure jump process with a compensator that is affine in $R$.
Without jumps this process is a Wishart-process (see also \eqref{eqn:wishart} and thereafter).
They were introduced by \cite{b91} and have been applied to many different fields such as term structure modeling and derivative pricing in \cite{gs03,gs04,dgt07,dgt08}.
The authors of \cite{bs07,bs11} consider multivariate stochastic volatility models for a class of matrix-valued Ornstein-Uhlenbeck processes driven by a L\`evy process of finite variation.
Closer to our subject is the work of \cite{dig09}.
There the authors investigate the power utility maximization problem in a multivariate Heston model where the covariation process follows a Wishart process.
They obtain the optimal portfolio and utility via a duality approach.

We start with the definition of multivariate affine stochastic volatility models.
Thereafter we carry out the martingale property of the stochastic exponential of some process which will allow us to prove optimality in the utility maximization problem.
More precisely we want to maximize expected utility of terminal wealth.
The wealth process $X^{x, \pi}$ is composed of the initial capital $x \in \R$ and gains from trading with strategy $\pi$ in the market.
We want to solve the problem in presence of random revenues $F$ which are paid at terminal time $T$, i.e.
\begin{align*}
V(x)= \sup_{\pi \in \mathcal A} \E \left[ U( X_T^{x,\pi} + F ) \right], \quad x \in \R,
\end{align*}
where $U$ is an exponential utility function.
Once a notion of admissibility is fixed we call any $\pi \in \mathcal A$ an admissible (trading) strategy.
Our aim is to explicitly describe the value function $V$ and the corresponding optimal strategy $\pi^{opt}$.
Similarly we examine the problem
\begin{align*}
V(x)= \sup_{\pi \in \mathcal A} \E \left[ U( X_T^{x,\pi} e^F ) \right], \quad x \in \R,
\end{align*}
where $U$ now is a power utility function.
It is known that the logarithmic utility maximization problem with $F=0$ can be solved explicitly for almost all semimartingale models, see e.g. \cite{gk03} and the references therein.
This is why we do not consider logarithmic utility in this work.


\subsection{Definition}


Multivariate affine stochastic volatility models are characterized via the joint Fourier-Laplace transform of the stochastic logarithm of the price process and the covariation process.
We suppose that the discounted $d$-dimensional asset price process $H$ is modeled as a stochastic exponential
\begin{align*}
  H_t= H_0 \mathcal E(N)_t, \quad t \in [0,T],
\end{align*}
where $N$ is the discounted $d$-dimensional logarithmic price process with $N_0=n \in \R^d$.
Let $R$ denote the stochastic covariation process with states in $S_d^+$ and starting in $R_0=r \in S_d^+$.

\begin{defn}[\cite{christa} Definition 5.3.5.]
 We call a stochastic process $(R,N)=(R_t,N_t)_{t \in [0,T]}$ with values in $S_d^+ \times \R^d$ a multivariate affine stochastic volatility model, if the following conditions are satisfied.
\begin{itemize}
  \item[(i)] The pair of processes $(R,N)$ is a stochastically continuous Markov process.
  \item[(ii)] Under the risk neutral measure $\mathbb Q$, the Fourier-Laplace transform of $(R, N)$ is  exponentially affine in the initial states $(r,n)$, i.e. there exist functions $(t,u,v) \mapsto \Psi(t,u,v)$ and $(t,u,v) \mapsto \Phi(t,u,v)$ such that
\begin{align}
\label{eqn:expaffinstochvolmod}
  \E^{\mathbb Q} \left[ e^{\operatorname{Tr} \left( u R_t\right) + v^{\top}  N_t}\right]
&=
\exp \left(  \operatorname{Tr}\left( \Psi(t,u,v) r \right) + v^{\top} n + \Phi(t,u,v) \right),
\end{align}
for all $(t,u,v) \in \mathcal Q$, where
\begin{align*}
  \mathcal Q = \left\{ (t,u,v) \in [0,T] \times S_d+\mathrm{i}S_d \times  \mathbb C^d \ : \ 
\E^{\mathbb Q} \left[ e^{\operatorname{Tr} \left( u R_t\right) + v^{\top}   N_t}\right]
< \infty
\right\}.
\end{align*}
	\item[(iii)] The asset price process $H$ is a martingale under the risk-neutral probability measure $\mathbb Q$.
\end{itemize}
\end{defn}

\subsection{The Martingale property}
\label{sec:martprop}

In the following sections it will play an important role under which conditions the stochastic exponential of a process involving an affine process becomes a true martingale.
This problem has applications in fields including absolute continuity of distributions of stochastic processes (see \cite{cfy05}) and the verification of optimality in utility maximization as we use it here.
On the state space $\R^n_+ \times \R^m$ and in a time-homogeneous setting the problem has been addressed already in \cite{km10,mms11}, which has then been extended to the state space $S_d^+ \times \R^d$ in \cite{christa}.

Suppose $R$ is an affine process with admissible parameter set $(\alpha, b,\beta^{ij},m,0)$ associated with truncation function $\chi^R$.
Let for all $s \in [0,T]$
\begin{align}
\label{eqn:ehoch}
 \int_{ \{|\operatorname{Tr}(\sigma_{\mu}(s) \xi)|>1\}} e^{\operatorname{Tr}(\sigma_{\mu}(s) \xi )} m(d\xi) 
& < \infty,
\end{align}
and consider the process 
\begin{align}
\label{eqn:P}
  P_t 
&= 
\int_0^t \sigma_Q^{\top}(s) \sqrt{R_s} dQ_s 
+ \int_0^t \operatorname{Tr} \left( \sigma_W(s) \sqrt{R_s} dW_s \right) 
+ \int_0^t \operatorname{Tr} \left( \sigma_{\hat Q}(s) \sqrt{R_s} d\hat Q_s \right) 
\\ \nonumber
& \quad
+ \int_0^t \int_{S_d^+ \setminus \{0\}} \left( e^{\operatorname{Tr}(\sigma_{\mu}(s) \xi)} -1 \right)d(\mu^R(ds,d\xi) - m(d\xi)ds),
t \in [0,T],
\end{align}
where $\sigma_Q:[0,T] \to \R^d$ and $\sigma_W, \sigma_{\hat Q},\sigma_{\mu}:[0,T] \to M^d$ are continuous functions of time.
The $d$-dimensional Brownian motion $Q$ is correlated to the matrix Brownian motion $W$ by
\begin{align*}
  dQ_t &=
dW_t \rho + \sqrt{1-\rho^{\top} \rho} dD_t.
\end{align*}
Here $D$ is a $d$-dimensional Brownian motion independent of $W$ and $\rho$ a $d$-dimensional vector with entries $\rho_i \in [-1,1]$, $i=1, \ldots,d$, satisfying $\rho^{\top} \rho \leq 1$.
The process $\hat Q$ is another independent $d \times d$-matrix Brownian motion.
Finally $\mu^R$ denotes the random measure associated to the jumps of $R$.
The goal of this section is to show that the stochastic exponential of $P$ is a martingale which will help us proving optimality in the utility maximization problems considered in the following sections.

\begin{thm}
\label{thm:martingale}
Assume \eqref{eqn:ehoch}, then the process $\mathcal E(P)$ is a martingale.
\end{thm}

It will be crucial that the conditional Fourier-Laplace transform of $(R,\hat P)$ with
\begin{align*} 
  \hat P&=\ln \left( \mathcal E (P)\right),
\end{align*}
is exponentially affine.
This is stated in the following Lemma.

\begin{lem}
\label{lem:laplaceexpaff}
The conditional Fourier-Laplace transform of $(R, \hat P)$ has an exponentially affine form.
More precisely, there exist functions $(s,t,u,v) \mapsto \Psi(s,t,u,v)$ and $(s,t,u,v) \mapsto \Phi(s,t,u,v)$ such that 
\begin{align}
\label{eqn:expaffinstochvolmod1}
  \E \left[ e^{\operatorname{Tr} \left( u R_t\right) + v \hat  P_t} \Big| (R_s, \hat P_s)\right]
&=
\exp \left(  \operatorname{Tr}\left( \Psi(s,t,u,v) R_s \right) + v \hat P_s + \Phi(s,t,u,v) \right),
\end{align}
for all $(s,t,u,v) \in \mathcal I$, where
\begin{align*}
  \mathcal I 
&= \Bigl\{ (s,t,u,v) \in [0,T]\times [0,T] \times S_d+\mathrm{i}S_d \times  \mathbb C \ : \ s \leq t, 
\E\left[ e^{\operatorname{Tr} \left( u R_t\right) + v \hat P_t} \Big| (R_s, \hat P_s)\right]
< \infty
\Bigr\}.
\end{align*}
The functions $\Phi$ and $\Psi$ have the following form
\begin{align}
\label{eqn:odephi}
  -\frac{\partial \Phi(s,t,u,v)}{\partial s} & = \mathscr F(s,\Psi(s,t,u,v),v), \quad \Phi(t,t,u,v)=0,
\\
\label{eqn:odepsi}
-\frac{\partial \Psi(s,t,u,v)}{\partial s} & = \mathscr R(s,\Psi(s,t,u,v),v), \quad \Psi(t,t,u,v)=u,
\end{align}
where
\begin{align*}
  \mathscr F(s,u,v)
&= \operatorname{Tr}(bu)
+ \int_{S_d^+ \setminus \{0\}} \left( e^{\operatorname{Tr}(u\xi)-v \operatorname{Tr}(\sigma_{\mu}(s) \xi)} - v e^{\operatorname{Tr}(\sigma_{\mu}(s) \xi)} + v -1 -\operatorname{Tr}(u \xi)
\right) m(d\xi),
\\
\mathscr R(s,u,v)
&=
2 u \alpha u 
+ B^*(u )
+ \frac12 v(v-1)\big(2 \sigma_Q(s)\rho^{\top}\sigma_W(s)+\sigma_W^{\top}(s)\sigma_W(s)+\sigma_{\hat Q}^{\top}(s)\sigma_{\hat Q}(s)
\\
& \quad
+\sigma_Q(s) \sigma_Q^{\top}(s)\big)
+ vu(\sigma_Q(s) \rho^{\top}+\sigma_W^{\top}(s))\Sigma
+v \Sigma^{\top} (\sigma_W(s) + \rho \sigma_Q^{\top}(s))u
,
\end{align*}
for all $(s,t,u,v) \in \mathcal I $.
\end{lem}

\begin{proof}
Since $\Delta P_t>-1$ for all $t \in [0,T]$ we have $\mathcal E(P)>0$ and hence the process $\hat P$ is well defined.
It\^o's formula 
gives
\begin{align}
\label{eqn:hatp} 
 \hat P_t 
&=
P_t - \frac12 \langle P, P \rangle_t + \sum_{s \leq t} \ln (1 + \Delta P_s) - \Delta P_s
\\ \nonumber
&=
P_t - \frac12 \langle P, P \rangle_t + \int_0^t \int_{S_d^+ \setminus \{0\}} \operatorname{Tr}(\sigma_{\mu}(s) \xi) (\mu^R(ds,d\xi)-m(d\xi)ds) 
\\ \nonumber
& \quad
+ \int_0^t \int_{S_d^+ \setminus \{0\}} \left( \operatorname{Tr}(\sigma_{\mu}(s) \xi) - e^{\operatorname{Tr}(\sigma_{\mu}(s) \xi)} + 1 \right)m(d\xi)ds.
\end{align}
Let us assume that there exist functions $\Psi:\mathcal I \to S_d+\mathrm{i}S_d$ and $\Phi: \mathcal I  \to \mathbb C$ such that the conditional Fourier-Laplace transform is of the form \eqref{eqn:expaffinstochvolmod1},
i.e. $\Psi(t,t,u,v)=u$ and $\Phi(t,t,u,v)=0$.
We consider $(s,t,u,v) \in \mathcal I$ from now on and denote
$$h(s,R_s , \hat P_s) = \exp\left( \operatorname{Tr}(\Psi(s,t,u,v) R_s) + v \hat P_s + \Phi(s,t,u,v)\right), \quad s\leq t.$$ 
Then we obtain with It\^o's formula for $s \leq t \in [0,T]$
\begin{align*}
&  dh(s,R_s,\hat P_s)
\\
&=
\frac{\partial h(s,R_s,\hat P_s)}{\partial s}  ds 
+ \sum_{i,j=1}^d \frac{\partial h(s,R_s, \hat P_s)}{\partial r_{ij}} dR_{ij,s} 
+ \frac{\partial h(s,R_s,\hat P_s)}{\partial  \hat p} d \hat P_s
\\
& \quad
+\frac{1}{2} \sum_{i,j,k,l=1}^d \frac{\partial^2 h(s,R_s, \hat P_s)}{\partial r_{ij} \partial r_{kl}} d \langle R_{ij},R_{kl} \rangle_s
+\frac{1}{2} \frac{\partial^2 h(s,R_s,\hat P_s)}{\partial \hat p \partial \hat p} d \langle \hat P, \hat P \rangle_s 
\\
& \quad
+ \sum_{i,j=1}^d \frac{\partial^2 h(s,R_s,\hat P_s)}{\partial r_{ij} \partial p} d \langle R_{ij},\hat P \rangle_s 
\\
& \quad 
+ \sum_{s \leq t} \biggl(
 h(s,R_s,\hat P_s) - h(s,R_{s-},\hat P_{s-}) 
- \sum_{i,j=1}^d \frac{\partial h(s,R_{s-},\hat P_{s-})}{\partial r_{ij}} \Delta R_{ij,s} 
- \frac{\partial h(s,R_{s-},\hat P_{s-})}{\partial \hat p} \Delta \hat P_s
\biggr).
\end{align*}
By the law of iterated expectations, for all $s \in [0,t]$ and $\bar s \in [s,t]$,
\begin{align*}
  h(s,R_s,\hat P_s) 
&=\E \left[ \E \l[ e^{\operatorname{Tr}(uR_t) +v \hat P_t} | (R_{\bar s}, \hat P_{\bar s}) \r] | (R_s,\hat P_s)\right] 
= \E \left[ h(\bar s , R_{\bar s}, \hat P_{\bar s}) | (R_s,\hat P_s)\right].
\end{align*}
This means that $h(\cdot,R,\hat P)$ is a martingale and hence the bounded variation part in the above equation has to be zero for all $s \leq t \in [0,T]$.
Straightforward computations then lead to
\begin{align}
\label{eqn:locmartingal}
0
&=
  h(s, R_s,\hat P_s) 
\Biggl(
\operatorname{Tr} \l( \frac{\partial \Psi(s,t,u,v)}{\partial s} R_s \r) 
  + \frac{\partial \Phi(s,t,u,v)}{\partial s}
+ 2\operatorname{Tr}(R_s \Psi(s,t,u,v)  \alpha \Psi(s,t,u,v))
\\ \nonumber
& \quad
+\operatorname{Tr}((b+B(R_s))\Psi(s,t,u,v))
+ \frac12 v(v-1) \operatorname{Tr}(R_s (2 \sigma_Q(s) \rho^{\top }\sigma_W(s) 
+\sigma_W^{\top}(s) \sigma_W(s)
\\ \nonumber
& \quad
+\sigma_{\hat Q}^{\top}(s) \sigma_{\hat Q}(s)
 + \sigma_Q(s) \sigma_Q^{\top}(s)))
+ 2 v \operatorname{Tr}( \Psi(s,t,u,v)(\sigma_Q(s) \rho^{\top} + \sigma_W(s))\Sigma)
\\ \nonumber
& \quad
+ \int_{S_d^+ \setminus\{0\}} \left( 
v\operatorname{Tr}(\sigma_{\mu}(s) \xi) 
- v e^{\operatorname{Tr}(\sigma_{\mu}(s)\xi)} 
+v 
+ e^{\operatorname{Tr}(\Psi(s,t,u,v)\xi) + v\operatorname{Tr}(\sigma_{\mu}(s) \xi)}-1
\right) m(d\xi) ds
\\ \nonumber
& \quad
+ \int_{S_d^+ \setminus\{0\}} 
\Bigl(
-\operatorname{Tr}(\Psi(s,t,u,v)\xi) -v\operatorname{Tr}(\sigma_{\mu}(s) \xi)
\Bigr) m(d\xi) ds
\Biggr).
\end{align}
Here we have used the representation \eqref{representationofx} of $R$, formula \eqref{eqn:hatp} and
\begin{align*}
d\langle R_{ij},R_{kl}\rangle_s
&=
A_{ijkl}(R_s),
\\
d\langle  \hat P, \hat P\rangle_s &=\operatorname{Tr} (R_s(2 \sigma_Q(s) \rho^{\top }\sigma_W(s) +\sigma_W^{\top}(s) \sigma_W(s) +\sigma_{\hat Q}^{\top}(s) \sigma_{\hat Q}(s) + \sigma_Q(s) \sigma_Q^{\top}(s))) ds,
\\
d \langle R, \hat P  \rangle_s &= 2 R_s ((\sigma_Q(s) \rho^{\top} + \sigma_W(s))\Sigma) ds, \quad s \in [0,T].
\end{align*}
We obtain the ODEs \eqref{eqn:odepsi} and \eqref{eqn:odephi} by equating coefficients.
There exists a unique local solution for equation \eqref{eqn:odepsi} since it is locally Lipschitz on $S_d+\mathrm{i}S_d$.
From \cite{sv10} Theorem 3.7 we know that the time before the ODE possibly explodes coincides with the expectation on the left hand side of \eqref{eqn:expaffinstochvolmod1} to exist.
This gives the result.
\end{proof}

In a time homogeneous setting \cite{christa} Theorem 5.3.4 gives necessary and sufficient conditions for the exponential of $e^{\tilde P}$ of a process $\tilde P \in \R^d$ to be a martingale.
We adapt the proof to our time-inhomogeneous process $\hat P$.

\begin{proof}[Proof of Theorem \ref{thm:martingale}]
We show that $e^{\hat P}$ is a martingale which gives the result due to $e^{\hat P}= \mathcal E(P)$.
Since \eqref{eqn:ehoch} gives
\begin{align*}
 \int_{ \{|\Delta \hat P_s|>1\}} e^{\Delta \hat P_s } m(d\xi) 
& = \int_{ \{|\operatorname{Tr}(\sigma_{\mu}(s) \xi)|>1\}} e^{\operatorname{Tr}(\sigma_{\mu}(s) \xi )} m(d\xi)  < \infty, 
\end{align*}
and $P$ is a local martingale, \cite{k04} Lemma 3.1 implies that $e^{\hat P}$ is a $\sigma$-martingale.
Due to $e^{\hat P}>0$ additionally we obtain by \cite{k04} Proposition 3.1 that $e^{\hat P}$ is a supermartingale.
Hence it is in particular integrable and thus $(s,t,0,1) \in \mathcal I$ for all $s \leq t \in [0,T]$.
From Lemma \ref{lem:laplaceexpaff} we have for all $s \leq t \in [0,T]$ 
\begin{align*}
  \E \left[ e^{\hat P_t} \Big| (R_s,\hat P_s)\right] 
&=
\exp \left( \operatorname{Tr}(\Psi(s,t,0,1)R_s) +\hat P_s + \Phi(s,t,0,1)\right).
\end{align*}
Thus, $\Psi(s,t,0,1)=0$ and $\Phi(s,t,0,1)=0$ for all $s \leq t \in [0,T]$ would ensure that $e^{\hat P}$ is a martingale.
Indeed Lemma \ref{lem:laplaceexpaff} implies
\begin{align*}
  \mathscr F(s,0,1)
&= \int_{S_d^+ \setminus \{0\}} \left( e^{-1 \operatorname{Tr}(\sigma_{\mu}(s) \xi)} - 1 e^{\operatorname{Tr}(\sigma_{\mu}(s) \xi)} +1-1 \right) m(d\xi) = 0,
\\
\mathscr R(s,0,1)
&=
 B^*(0 )
+ \frac12 1(1-1)(2 \sigma_Q(s)\rho^{\top}\sigma_W(s)+\sigma_W^{\top}(s)\sigma_W(s)+\sigma_{\hat Q}^{\top}(s)\sigma_{\hat Q}(s)+\sigma_Q(s) \sigma_Q^{\top}(s)) =0,
\end{align*}
and hence the solution of 
\begin{align*}
  -\frac{\partial \Phi(s,t,0,1)}{\partial s} & = \mathscr F(s,\Psi(s,t,0,1),1)=0, \quad \Phi(t,t,0,1)=0,
\\
-\frac{\partial \Psi(s,t,0,1)}{\partial s} & = \mathscr R(s,\Psi(s,t,0,1),1)=0, \quad \Psi(t,t,0,1)=0,
\end{align*}
is $\Psi(s,t,0,1)=0$ and $\Phi(s,t,0,1)=0$ for all $s \leq t \in [0,T]$.
\end{proof}


\subsection{Solution in a continuous multivariate affine stochastic volatility model}
\label{sec:heston}

In this section we introduce a continuous affine stochastic volatility model which is a natural multivariate extension of the Heston model.
We then formulate the utility maximization problem in this model and solve it for power and exponential utility.
This allows to describe the optimal strategy and maximal expected utility in terms of the model parameters and a Riccati ODE.

Assume that there exists a financial market with one riskless bond with zero interest rate and $d$ risky assets $H=(H_1, \ldots , H_d)$. 
The process $H$ is modeled as stochastic exponential $H=H_0 \mathcal E(N)$, where $N=(N_1, \ldots, N_d)$ is given by 
\begin{align}
\label{price}
 dN_t 
& = R_t \eta  dt + \sqrt{R_t} d Q_t,  \quad t \in [0,T],
\end{align}
with $Q$ being a vector Brownian motion with values in $\R^d$ and $\eta \in \R^d$ being a constant vector.
The stochastic volatility process is continuous and affine with admissible parameter set $(\alpha,b,\beta^{ij},0,0)$, i.e. it satisfies the SDE
\begin{align}
\label{risk}
 dR_t= (b + B(R_t)) dt + \sqrt{R_t} dW_t \Sigma + \Sigma^{\top} dW_t^{\top} \sqrt{R_t}, \quad R_0= r \in S_d^+, \quad t \in [0,T],
\end{align}
where $W$ is a $d \times d$ matrix-valued Brownian motion.
The Brownian motion $Q$ driving the assets returns and the Brownian motion $W$ of the stochastic covariation matrix are allowed to be correlated in a certain way.
Let $\rho=(\rho_1, \ldots,\rho_d)^{\top}$ be a vector with entries $\rho_i \in [-1,1]$, $i=1,\ldots ,d$, and such that $\rho^{\top} \rho \leq 1$.
With a $d$-dimensional Brownian motion $D$ independent of $W$ we can write
\begin{align*}
dQ_t=dW_t \rho + \sqrt{1-\rho^{\top} \rho} dD_t.
\end{align*}
Hence the correlation between the scalar Brownian motions $Q^i$ and $W^{mn}$ is given by $\rho_n$ if $i=m$ and else it is $0$.
The structure of the correlation between $Q$ and $W$ has been chosen in this way in order to ensure the model to be affine, which can be seen in the following Proposition.

\begin{prop}
\label{prop:contstochvol}
The pair of processes $(R,N)$ is a multivariate stochastic volatility model with functions $\Phi$ and $\Psi$ solving
\begin{align}
\label{eqn:laplacecont}
\frac{\partial  \Phi(t,u,v)} {\partial t} &=  b \Psi(t,u,v),
\\ \nonumber
\Phi(0,u,v)&=0,
\\ \nonumber
  \frac{\partial \Psi(t,u,v)}{\partial t}
  &=
  2 \Psi(t,u,v) \alpha \Psi(t,u,v) + B^*(\Psi(t,u,v)) 
+ \left( v- \eta \right) \rho^{\top} \Sigma \Psi(t,u,v) 
\\ 
\label{eqn:laplacecont2}
& \quad
+ \Psi(t,u,v) \Sigma^{\top} \rho \left( v - \eta\right)^{\top} +\frac12 v v^{\top},
\\\nonumber
\Psi(0,u,v)&=u,
\end{align}
for all $(t,u,v) \in \mathcal Q$.
\end{prop}

\begin{proof}
That $(R,N)$ is stochastically continuous follows by definition.
From Theorem \ref{thm:martingale} we obtain with $\sigma_W \equiv \sigma_{\hat Q} \equiv \sigma_\mu \equiv 0$ and $\sigma_Q \equiv -\eta$ that the process $\mathcal E (X)$, where $X_t= - \int_0^{t} \eta^{\top} \sqrt{R_s} dQ_s$, $t \in [0,T]$, is a martingale.
This means that we can define a new probability measure $\mathbb Q$ by $d \mathbb Q = \mathcal E \left( X \right)_T d\mathbb P$.
The theorems of Girsanov and L\'evy then give that $\tilde Q= Q + \int_0^{\cdot} \sqrt{R_s} \eta ds$ is a $\mathbb Q$-Brownian motion.
Hence $N=\int_0^{\cdot} \sqrt{R_s} d\tilde Q_s$ is a local $\mathbb Q$-martingale and again, according to Lemma 3.1 and Proposition 3.1 in \cite{k04} the asset price process $H=\mathcal E \left( N \right)$ is a $\mathbb Q$-martingale.

It remains to prove formula \eqref{eqn:expaffinstochvolmod}.
Since for all $t \in [0,T]$, the covariation $d\langle W_{ij},X\rangle_t= - (\sqrt{R_t} \eta)_i \rho_j dt$, $i,j=1, \ldots,d$,  
we have by Girsanov's theorem that 
\begin{align*}
  \tilde W=W+ \int_0^{\cdot} \sqrt{R_s} \eta \rho^{\top} ds, 
\end{align*}
is a Brownian motion.
Hence the dynamics of \eqref{price} and \eqref{risk} under $\mathbb Q$ can be written as
\begin{align*}
dN_t &=
\sqrt{R_t} d \tilde Q_t ,
\\
 dR_t&= (b + B(R_t)- R_t \eta \rho^{\top} \Sigma - \Sigma^{\top} \rho \eta^{\top} R_t) dt + \sqrt{R_t} d \tilde W_t \Sigma + \Sigma^{\top} d \tilde W_t^{\top} \sqrt{R_t}.
\end{align*}

Consider the conditional Fourier-Laplace transform $\E^{\mathbb Q} \left[ e^{\operatorname{Tr} \left( u R_t\right) + v^\top  N_t} | (R_s,N_s)\right]$ under the measure $\mb Q$ for $s \leq t$.
Since $R$ and $N$ are time-homogeneous processes, we have that for $(t,u,v) \in \mathcal Q$ the process 
\begin{align*}
  h(t-s,R_s,N_s) &=\E^{\mathbb Q} \left[ e^{\operatorname{Tr} \left( u R_t\right) + v^\top  N_t} | (R_s,N_s)\right]
\end{align*}
is a martingale.
We conjecture that the conditional Fourier-Laplace transform is of exponentially affine form, more precisely that there exist functions $\Psi:\mathcal Q \to S_d+\mathrm{i}S_d$ and $\Phi:\mathcal Q \to \mathbb C^d$ such that $h$ has the form
\begin{align*}
  h(t-s,R_s,N_s)&=
\exp \left(  \operatorname{Tr}\left( \Psi(t-s,u,v) R_s \right) + v^\top N_s + \Phi(t-s,u,v) \right).
\end{align*}
Applying It\^o's formula to $h(t-s,R_s,N_s)$ and taking into account that it is a martingale, the bounded variation term needs to satisfy
\begin{align*}
  0&=
  \operatorname{Tr} \l( \frac{\partial \Psi(t-s,u,v)}{\partial s} R_s \r)
+ \frac{\partial \Phi(t-s,u,v)}{\partial s}
+ 2 \operatorname{Tr}(R_s \Psi(t-s,u,v) \alpha \Psi(t-s,u,v)) 
+ \frac12 v^\top R_s v
\\
& \quad
+ \operatorname{Tr}\left( (b + B(R_s) - R_s \eta \rho^{\top} \Sigma - \Sigma^{\top} \rho \eta^{\top} R_s) \Psi(t-s,u,v) \right)
+ 2 \operatorname{Tr}(\Psi(t-s,u,v) R_s \Sigma^\top \rho v^\top),
\end{align*}
for $s \leq t$.
Here we have used the fact that $d \langle R_{ij}, N_k\rangle_t = 2 (\rho \Sigma)_i R_{t,jk} dt$ for $i,j,k=1,\ldots,d$.
Identifying coefficients leads to the following system of ODEs
\begin{align*}
-\frac{\partial \Psi(t-s,u,v)}{\partial s} 
&= 
2 \Psi(t-s,u,v) \alpha \Psi(t-s,u,v) + B^*(\Psi(t-s,u,v)) 
\\ \nonumber
& \quad
+ \left( v - \eta \right) \rho^{\top} \Sigma \Psi(t-s,u,v) 
+ \Psi(t-s,u,v) \Sigma^{\top} \rho \left( v - \eta\right)^{\top} +\frac12 v v^{\top}, 
\\ \nonumber
-\frac{\partial \Phi(t-s,u,v)}{\partial s} &=  b \Psi(t-s,u,v),
\end{align*}
with boundary conditions
\begin{align*}
  \Psi(t,u,v)=u, \quad \Phi(t,u,v)=0.
\end{align*}
With a change of variable and since $\Phi$ can be determined via integration as soon as $\Psi$ is known, we derive \eqref{eqn:laplacecont},
provided there exists a solution to the ODE for $\Psi$, i.e.
\begin{align*}
  \frac{\partial \Psi(t,u,v)}{\partial t} &= \theta(\Psi(t,u,v),v), \quad 
\Psi(0,u,v)=u,
\end{align*}
with
\begin{align*}
\theta(u,v)&= 2 u\alpha u + B^*(u) 
+ (v-\eta) \rho^{\top} \Sigma u 
+  u \Sigma^{\top} \rho (v - \eta )^{\top} 
+\frac12 vv^{\top} ,
\end{align*}
for all $(t,u,v) \in \mathcal Q$.
Since $\theta$ is locally Lipschitz, there exists a unique local solution $\Psi$ of \eqref{eqn:laplacecont2}. 
From \cite{sv10} Theorem 3.7 we know that the time before the ODE possibly explodes coincides with the expectation on the left hand side of \eqref{eqn:expaffinstochvolmod} to exist.
This gives the result.
\end{proof}

\begin{rmk}
\label{rmk:ordexp}
  We have chosen to model the asset price process $H$ as stochastic exponential of $N$.
It is also possible to model $H$ as ordinary exponential, i.e. $H= H_0 e^N$.
In this case we have $H=H_0 \mathcal E (\tilde N)$ with $d\tilde N_t= \sqrt{R_t} dQ_t + R_t (\eta + \frac12)dt$, $t \in [0,T]$.
Hence we are back in the setting considered above.
\end{rmk}

\subsubsection{Power utility}

We set $F=0$ and assume that the investor's preferences are described by the power utility function
 \begin{align*}
  U(x) = \frac{1}{\gamma} x^{\gamma}, \quad \quad x \geq 0, \ \gamma \in (0,1).
 \end{align*}
Let $\mathcal A$ be the set of all $d$-dimensional predictable processes $\pi$ that satisfy $\int_0^T \pi_s^{\top} \pi_s ds < \infty$ a.s.
Note that $\pi_i$ denotes the fraction of the wealth invested in stock $i$, where $i=1, \ldots, d$.
Any process $\pi \in \mathcal A$ is called an admissible (trading) strategy.
Under these assumptions the wealth process evolves as follows
\begin{align*}
 X_t^{x,\pi} 
&=
x+ \int_0^t X_s^{x, \pi} \pi_s^{\top} dN_s = x + \int_0^t X_s^{x, \pi} \pi_s^{\top} R_s \eta \ ds + \int_0^t X_s^{x, \pi} \pi_s^{\top} \sqrt{R_s} d Q_s, 
\end{align*}
for $t \in [0,T]$.
It can also be written as stochastic exponential
\begin{align*}
 X_t^{x,\pi} = x \mathcal E \left( \int_0^t \pi_s^{\top} R_s \eta ds + \int_0^t \pi_s^{\top} \sqrt{R_s} d Q_s \right), \quad t \in [0,T].
\end{align*}
As described earlier the investor wants to maximize their expected utility of terminal wealth. 
In order to model interest and exchange rates later, we want to take the function
\begin{align*}
 F(O_T) &= Tr(aO_T),
\end{align*}
into account, where $a$ is a $d \times d$-matrix and $O_T$ the final value of the process
\begin{align}
\label{eqn:Ot}
   O_t &= \int_0^t \sigma \sqrt{R_s} d\hat Q_s + \int_0^t (o_1 +o_2 R_s)ds, \quad t \in [0,T].
\end{align}
Here $o_1,o_2 ,\sigma \in M_d$ and $\hat Q$ is a $d \times d$-dimensional Brownian motion independent of the Brownian motions $W$ and $Q$.
From now on we will write $F^{a,\sigma,o_1,o_2}$ for $F(O_T)$ with $O_T$ given by \eqref{eqn:Ot} to depict the structure of $F$ in more detail.
We want to solve the maximization problem
\begin{align}
\label{maxprobpower} 
V^{a,\sigma,o_1,o_2}(x)= \sup_{\pi \in \mathcal A} \E \left[ \frac{1}{\gamma} \Big(X_T^{x,\pi} \exp(F^{a,\sigma,o_1,o_2})\Big)^{\gamma}  \right], \quad x \geq 0.
\end{align}
Our main result describes the value function and the optimal strategy explicitly in terms of the model parameters.
\begin{thm}
\label{thmpower}
Let the linear diffusion term $\alpha$ belong to $S_d^{++}$ and suppose the linear drift term
$B$ in \eqref{risk} is of the form $B(r)=r\hat B  +  \hat B^{\top}r$, $r \in S_d^+$, with $\hat B \in M_d$.
Define the matrix-valued functions $A:[0,T] \to M_{2d}$ and $A_{ij}:[0,T] \to M_d$, $i,j =1, \ldots,d$, by
\begin{align}
\label{eqn:Apower}
  A(t)
&=\left( \begin{matrix}
             A_{11}(t)  & A_{12}(t) \\
  	     A_{21}(t)	& A_{22}(t)
             \end{matrix}
\right)
= \exp \left( (T-t) \left( \begin{matrix}
                            \frac{\gamma}{1-\gamma} \Sigma^{\top} \rho \eta^{\top} + \hat B^{\top} &  -2 \alpha -\frac{2 \gamma}{1-\gamma} \Sigma^{\top} \rho \rho^{\top} \Sigma \\
			     \frac12 \sigma^\top a a^\top \sigma +\frac{\gamma}{2(1- \gamma)} \eta \eta^\top + a o_2	&	- \frac{\gamma}{1-\gamma} \eta \rho^{\top} \Sigma - \hat B 
                          \end{matrix} \right)
\right).
\end{align}

\noindent
Then the value function and the corresponding optimal strategy are given by
\begin{align*}
  V^{a,\sigma,o_1,o_2}(x)
&=
\frac{1}{\gamma} x^{\gamma} 
\exp\left(
	 \operatorname{Tr}(  A_{22}^{-1}(0)  A_{21}(0) r ) + \int_0^T \operatorname{Tr}\l(A_{22}^{-1}(s)  A_{21}(s) b + ao_1 \r) ds
\right),
\\
\pi^{opt}_t&= \frac{1}{1-\gamma} \left( \eta + 2 A_{22}^{-1}(t)  A_{21}(t) \Sigma^{\top} \rho\right) , \quad x \geq 0, \ t \in [0,T].
\end{align*}
\end{thm}

Before we prove the above theorem we motivate the present approach which can also be found in \cite{HuImkellerMuller}.
Note that in contrast to \cite{HuImkellerMuller} the coefficients in the evolution of $N$ are not bounded.
We solve this problem by using the martingale optimality principle, in particular we aim to construct processes $L^{\pi}$ as well as a strategy $\pi^{opt}$ such that
\begin{itemize}
 \item $L^{\pi}_T = U(X_T^{x,\pi} \exp(F^{a,\sigma,o_1,o_2}))$ for all $\pi \in \mathcal A$,
 \item $L^{\pi}$ is a supermartingale for all $\pi \in \mathcal A$ and there is a particular strategy $\pi^{opt} \in \mathcal A$ such that $L^{\pi^{opt}}$ is a martingale.
\end{itemize}
Note that our assumptions on the filtration then imply that $L_0^{\pi}=C$ for all $\pi \in \mathcal A$ and a constant $C > 0$.
Applying the utility function to $X^{x, \pi}_T \exp(F^{a,\sigma,o_1,o_2})$ we get
\begin{align*}
& \frac{1}{\gamma} \Big( X_T^{x, \pi} \exp(F^{a,\sigma,o_1,o_2}) \Big)^{\gamma} 
\\
&= 
\frac{1}{\gamma} x^{\gamma} \exp \left( \int_0^T \gamma \pi_s^{\top} R_s \eta ds + \int_0^T \gamma \pi_s^{\top} \sqrt{R_s} d Q_s -\frac12 \int_0^T \gamma \pi_s^{\top} R_s \pi_s ds + \gamma F^{a,\sigma,o_1,o_2}\right).
\end{align*}
This suggests the following choice of $L^{\pi}$
\begin{align*}
 L_t^{\pi} = x^{\gamma} \exp \left( \int_0^t \gamma \pi_s^{\top} R_s \eta ds + \int_0^t \gamma \pi_s^{\top} \sqrt{R_s} d Q_s -\frac12 \int_0^t \gamma \pi_s^{\top} R_s \pi_s ds + Y_t \right),
\end{align*}
where $Y$ is the first component of the solution of a BSDE with terminal condition $\gamma F^{a,\sigma,o_1,o_2}$.
More precisely we want to find a generator $f$ for the BSDE
\begin{align}
  \label{BSDEpower} 
Y_t 
&= 
\gamma F^{a,\sigma,o_1,o_2} - \int_t^T \operatorname{Tr}( Z^{\top}_s dW_s) - \int_t^T \operatorname{Tr}( \hat Z^{\top}_s d \hat Q_s) + \int_t^T f(R_s,Z_s,\hat Z_s) ds, \quad t \in [0,T],
\end{align}
such that its solution $(Y,Z,\hat Z)$ implies that $L^{\pi}$ meets the above requirements.

\begin{lem}
\label{lem:supermartpower}
Let $\alpha \in S_d^{++}$, $B(r)=r\hat B  +  \hat B^{\top}r$ with $\hat B \in M_d$ and recall \eqref{eqn:Apower}.
If the generator $f: S_d^+ \times M_d \to \R$ is of the form 
\begin{align}
\label{generatorpower}
 f(r,z, \hat z) &= \frac12 \operatorname{Tr}(zz^{\top}) + \frac12 \operatorname{Tr}(\hat z \hat z^{\top}) + \frac{\gamma}{2(1-\gamma)} | \sqrt{r} \eta + z \rho|^2, \quad (r,z,\hat z) \in S_d^+ \times M_d \times M_d ,
\end{align}
then \eqref{BSDEpower} is solved by 
\begin{align}
\label{eqn:BSDEsolutionpower}
Y_t & = 
\operatorname{Tr}(A_{22}^{-1}(t)  A_{21}(t) R_t)
+ \operatorname{Tr} (a O_t)
+ \int_{t}^T \operatorname{Tr} \l(A_{22}^{-1}(s)  A_{21}(s) b + ao_1\r) ds ,
\\ \nonumber
Z_t & = 2 \sqrt{R_t} A_{22}^{-1}(t)  A_{21}(t) \Sigma^{\top},
\\ \nonumber
\hat Z_t & = \sqrt{R_t} \sigma^\top a, 
\qquad t \in [0,T].
\end{align}
Moreover $L^{\pi}$ is a supermartingale for every strategy $\pi \in \mathcal A$ and for
\begin{align}
\label{optstratpower}
\pi^{opt}_t & = \frac{1}{1- \gamma} \left( \eta + 2 A_{22}^{-1}(t) A_{21}(t) \Sigma^{\top} \rho \right) , \quad t \in [0,T],
\end{align}
the process $L^{\pi^{opt}}$ is a martingale.
\end{lem}

\begin{proof}
Let us define the constants
$$c_{zz}= \frac12 I_d + \frac{\gamma}{2(1-\gamma)}\rho \rho^{\top}, \quad
c_{\hat z \hat z}= \frac12 I_d, \quad c_{z \sqrt x}=\frac{\gamma}{2(1-\gamma)} \rho \eta^{\top}, \quad c_x=\frac{\gamma}{2(1-\gamma)}\eta \eta^{\top}.$$
Note that $c_{zz}$ is positive definite.
By Proposition \ref{easyriccati} we know that the ODE 
\begin{align*}
- \frac{d\Gamma(t)}{d t} &=
 \Gamma(t) \Sigma^{\top} c_{zz} \Sigma \Gamma(t)  \nonumber
+ B^*(\Gamma(t)) + 2\Gamma(t)\Sigma^{\top}c_{z \sqrt x}+ 2c_{z \sqrt x}^{\top} \Sigma \Gamma(t) + \frac12 \sigma^\top a a^\top \sigma +c_x + a o_2, 
\\
\quad \Gamma(T)&=0 , \nonumber
\end{align*}
has the solution $\Gamma(t) = A_{22}^{-1}(t)  A_{21}(t)$, $t \in [0,T]$.
We then obtain from Theorem \ref{thm:matrixBSDEsolution} that the BSDE \eqref{BSDEpower} with generator \eqref{generatorpower} is solved by \eqref{eqn:BSDEsolutionpower}.

We show the local (super)martingale property for $L^\pi$.
  With It\^o's formula applied to $L^{\pi}$ we have for all $\pi \in \mathcal A$
  \begin{align*}
&   dL_t^{\pi} 
\\
& =
L_t^{\pi} \left( 
\gamma \pi^{\top}_t \sqrt{R_t} dQ_t 
+ \operatorname{Tr}(Z_t^{\top} dW_t) 
+ \operatorname{Tr}( \hat Z^{\top}_t d \hat Q_t) 
\right) 
+ L_t^{\pi} \left( \gamma \pi^{\top}_t R_t \eta - \frac12 \gamma \pi^{\top}_t R_t \pi_t - f(R_t,Z_t,\hat Z_t) \right) dt
\\
& \quad 
+ \frac12 L_t^{\pi} \left( \left|\gamma \sqrt{R_t} \pi_t \rho^{\top} + Z_t\right|^2 + \left|\gamma \sqrt{1 - \rho^{\top} \rho}\pi^{\top}_t \sqrt{R_t}\right|^2 
+ \operatorname{Tr} (\hat Z_t \hat Z_t^\top)
\right) dt,
  \end{align*}
where we have used $dQ_t=dW_t \rho + \sqrt{1-\rho^{\top} \rho} dD_t$.
If the finite variation part satisfies $dt \otimes \mathbb P$-a.e.
\begin{align*}
&  L_t^{\pi} \left( \gamma \pi_t^{\top} R_t \eta - \frac12 \gamma \pi_t^{\top} R_t \pi_t - f(R_t,Z_t,\hat Z_t)  
+ \frac12 \left|\gamma \sqrt{R_t} \pi_t \rho^{\top} + Z_t\right|^2 + \frac12 \left|\gamma \sqrt{1 - \rho^{\top} \rho}\pi_t^{\top} \sqrt{R_t}\right|^2 
+ \frac12 |\hat Z_t|^2 \right)  
\leq 0 , 
\end{align*}
then we know that $L^\pi$ is a local supermartingale.
Indeed, since $L_t^\pi >0$ for all $t \in [0,T]$ we only need to check whether
\begin{align*}
 - f(R_t,Z_t,\hat Z_t) 
& \leq 
-\gamma \pi^{\top}_t R_t \eta 
+ \frac12 \gamma \pi^{\top}_t R_t \pi_t
- \frac12 \left|\gamma \sqrt{R_t} \pi_t \rho^{\top} + Z_t\right|^2
- \frac12 \left|\gamma \sqrt{1 - \rho^{\top} \rho}\pi^{\top}_t \sqrt{R_t}\right|^2
- \frac12 |\hat Z_t|^2 .
\end{align*}
This is equivalent to
\begin{align*}
-  f(R_t,Z_t,\hat Z_t) 
&\leq 
-\gamma \pi_t^{\top} R_t \eta 
+ \frac12 \gamma \pi^{\top}_t R_t \pi_t  
- \frac12 \left|Z_t\right|^2 
- \gamma \operatorname{Tr}( \sqrt{R_t} \pi_t \rho^{\top} Z_t^{\top})  
- \frac12 \gamma^2 \pi^{\top}_t R_t \pi_t 
- \frac12 |\hat Z_t|^2 
\\
& = 
\frac12 \gamma (1-\gamma) \left|\sqrt{R_t} \pi_t - \frac{1}{1-\gamma} ( \sqrt{R_t}\eta + Z_t \rho )\right|^2
-\frac{\gamma}{2(1-\gamma)} \left| \sqrt{R_t}\eta + Z_t \rho  \right|^2 
- \frac12 |Z_t|^2
-\frac12 |\hat Z_t|^2 .
\end{align*}
If we use \eqref{generatorpower}, we see that this inequality is true for every $\pi \in \mathcal A$.
For $\pi^{opt}$ from \eqref{optstratpower} and applying the particular form of $Z$, the above inequality turns out to be an equality and hence the process $L^{\pi^{opt}}$ is a local martingale.
Note that $\pi^{opt} \in \mathcal A$.

We proceed showing that $L^{\pi}$ is a true supermartingale for all $\pi \in \mathcal A$.
By definition there exists a sequence of stopping times $(\tau_n)_{n \in \mathbb N}$ converging to $T$ such that $L^{\pi}_{\cdot \wedge \tau_n}$ is a supermartingale.
Since $L^{\pi}$ is bounded below by zero we may use Fatou's Lemma to pass to the limit:
$$\mathbb E \left[ L_t^{\pi } | \mathcal F_s \right] 
= \mathbb E \left[ \lim_{n \rightarrow \infty} L^{\pi}_{t \wedge \tau_n} | \mathcal F_s \right] 
\leq \lim_{n \rightarrow \infty} \mathbb E \left[  L^{\pi}_{t \wedge \tau_n} | \mathcal F_s \right]
\leq \lim_{n \rightarrow \infty} L_{s \wedge \tau_n}^{\pi}
= L_s^{\pi}, \quad s \leq t \in  [0,T]. $$
Note that from \eqref{eqn:BSDEsolutionpower} and \eqref{optstratpower} we have
\begin{align*}
  L^{\pi^{opt}}_t &= x^{\gamma} \mathcal E \left( \gamma \int_0^t (\pi^{opt})^{\top}_s \sqrt{R_s} dQ_s + \int_0^t \operatorname{Tr}\left(2 \Sigma (A_{22}^{-1}(s) A_{21}(s))^{\top} \sqrt{R_s} dW_s \right)
+ \int_0^t \operatorname{Tr}\left( a^\top \sigma \sqrt{R_s} d\hat Q_s \right) \right).
\end{align*}
By choosing $\sigma_Q(s)= \gamma \pi^{opt}_s$, $\sigma_W(s)=2 \Sigma (A_{22}^{-1}(s) A_{21}(s))^{\top}$, $\sigma_{\hat Q}(s) \equiv a^\top \sigma $ and $\sigma_{\mu} \equiv 0$, $s \in [0,T]$, we derive from Theorem \ref{thm:martingale} that $L^{\pi^{opt}}$ is a true martingale.
\end{proof}

\begin{proof}[Proof of Theorem \ref{thmpower}.]
Note that we derive from Lemma \ref{lem:supermartpower} for all $ \pi \in \mathcal A$
\begin{align*}
\mathbb E\left[U(X_T^{x,\pi} \exp(F^{a,\sigma,o_1,o_2}))\right] 
 =\mathbb E \left[\frac{1}{\gamma} L_T^{\pi} \right] 
& \leq \mathbb E \left[ \frac{1}{\gamma} L_0^{\pi} \right] 
= \frac{1}{\gamma} x^{\gamma} \exp(Y_0). 
\end{align*}
The strategy $\pi^{opt}$ is indeed optimal since we have that $L^{\pi^{opt}}$ is a martingale and hence
$$\mathbb E\left[U(X_T^{x,\pi^{opt}}\exp(F^{a,\sigma,o_1,o_2}))\right] = \mathbb E \left[ \frac{1}{\gamma} L_0^{\pi^{opt}} \right] .$$
This immediately gives the value function.
\end{proof}

\begin{rmk}
In dimension $d=1$, for the case $F(O_T)=0$ and with a slightly different choice of parameters, this result was derived by the authors of \cite{km08}.
They represent the optimal strategy in terms of an opportunity process and use semimartingale characteristics.
In our setting the opportunity process is $e^Y$, see also \cite{n10} and in particular \cite{hmuw10} for a survey on the relationship between BSDEs and duality methods in utility maximization.
On a heuristic level the result for $d=1$ and $F(O_T)=0$ appears in \cite{l07}.
Also using duality methods \cite{dig09} derive a result similar to Theorem \ref{thmpower}.
\end{rmk}

Finally we are able to give the indifference value of change of numeraire in two examples. 
Let us first look at the special situation where 
\begin{align*}
F^{-I_d,0,o_3,0}&=
- \operatorname{Tr}(o_3) T 
 \quad \mbox{or} \quad
F^{-I_d,0,o_1,o_2}=
- \int_0^T \operatorname{Tr}(o_1 + o_2 R_s) ds, \quad o_1, \ldots, o_3 \in M_d,
\end{align*}
and understand this as the possibly stochastic discounting of the investors terminal wealth. 
The indifference value $p$ of changing between those two numeraires is then defined by 
\begin{align*}
 V^{-I_d,0,o_1,o_2}(x-p(x)) &= V^{-I_d,0,o_3,0}(x).
\end{align*}

\begin{prop}
\label{prop:powerind}
 The indifference value of changing from a fixed interest rate $F^{-I_d,0,o_3,0}$ to the floating one $F^{-I_d,0,o_1,o_2}$ is
\begin{align*}
 p(x) &= x - x \exp \l( \frac{1}{\gamma} \l( 
\operatorname{Tr}(  B_{22}^{-1}(0)  B_{21}(0) r ) + \int_0^T \operatorname{Tr}\l(B_{22}^{-1}(s)  B_{21}(s) b - o_3 \r) ds 
\right.\right. 
\\  & \hspace{2cm} 
\left.\left.
- \operatorname{Tr}(  A_{22}^{-1}(0)  A_{21}(0) r ) - \int_0^T \operatorname{Tr}\l(A_{22}^{-1}(s)  A_{21}(s) b - o_1 \r) ds \r) \r) ,
\end{align*}
with
\begin{align*}
\left( \begin{matrix}
             A_{11}(t)  & A_{12}(t) \\
  	     A_{21}(t)	& A_{22}(t)
             \end{matrix}
\right)
&
= \exp \left( (T-t) \left( \begin{matrix}
                    \frac{\gamma}{1-\gamma} \Sigma^{\top} \rho \eta^{\top} + \hat B^{\top} 
		    &  -2 \alpha -\frac{2 \gamma}{1-\gamma} \Sigma^{\top} \rho \rho^{\top} \Sigma \\
			     \frac{\gamma}{2(1- \gamma)} \eta \eta^\top - o_2	&	- \frac{\gamma}{1-\gamma} \eta \rho^{\top} \Sigma - \hat B 
                          \end{matrix} \right)
\right)
\\
\left( \begin{matrix}
             B_{11}(t)  & B_{12}(t) \\
  	     B_{21}(t)	& B_{22}(t)
             \end{matrix}
\right)
&
= \exp \left( (T-t) \left( \begin{matrix}
                            \frac{\gamma}{1-\gamma} \Sigma^{\top} \rho \eta^{\top} + \hat B^{\top} &  -2 \alpha -\frac{2 \gamma}{1-\gamma} \Sigma^{\top} \rho \rho^{\top} \Sigma \\
			     \frac{\gamma}{2(1- \gamma)} \eta \eta^\top	&	- \frac{\gamma}{1-\gamma} \eta \rho^{\top} \Sigma - \hat B 
                          \end{matrix} \right)
\right).
\end{align*}

\end{prop}

\begin{proof}
 By Theorem \ref{thmpower} we have that 
\begin{align*}
 V^{-I_d,0,o_1,o_2}(x-p(x)) &=
\frac{1}{\gamma} (x-p(x))^\gamma \exp\l(  \operatorname{Tr}(  A_{22}^{-1}(0)  A_{21}(0) r ) + \int_0^T \operatorname{Tr}\l(A_{22}^{-1}(s)  A_{21}(s) b -o_1 \r) ds \r)
\\
 V^{-I_d,0,o_3,0}(x) &=
\frac{1}{\gamma} x^\gamma \exp\l(  \operatorname{Tr}(  B_{22}^{-1}(0)  B_{21}(0) r ) + \int_0^T \operatorname{Tr}\l(B_{22}^{-1}(s)  B_{21}(s) b - o_3 \r) ds \r),
\end{align*}
which gives the result.
\end{proof}

In a similar way we can describe the indifference value of change of numeraire from a fixed exchange rate
\begin{align*}
 F^{I_d,0,o_3,0}&=
 \operatorname{Tr}(o_3) T,
\end{align*}
with $o_3 \in M_d$ to a random valued exchange rate
\begin{align*}
 F^{a,\sigma,o_1,o_2}&= \operatorname{Tr}(aO_T),
\end{align*}
with $a \in M_d$ and $O_T$ from \eqref{eqn:Ot}.

\begin{prop}
 The indifference value of changing from a fixed exchange rate $F^{-I_d,0,o_3,0}$ to the floating one $F^{a,\sigma,o_1,o_2}$ is
\begin{align*}
 p(x) &= x - x \exp \l( \frac{1}{\gamma} \l( \operatorname{Tr}(  B_{22}^{-1}(0)  B_{21}(0) r ) + \int_0^T \operatorname{Tr}\l(B_{22}^{-1}(s)  B_{21}(s) b + o_3 \r) ds 
\right. \right. 
\\  & \hspace{2cm} 
\left.\left.
- \operatorname{Tr}(  A_{22}^{-1}(0)  A_{21}(0) r ) - \int_0^T \operatorname{Tr}\l(A_{22}^{-1}(s)  A_{21}(s) b +ao_1 \r) ds \r) \r) ,
\end{align*}
with
\begin{align*}
\left( \begin{matrix}
             A_{11}(t)  & A_{12}(t) \\
  	     A_{21}(t)	& A_{22}(t)
             \end{matrix}
\right)
&
= \exp \left( (T-t) \left( \begin{matrix}
                    \frac{\gamma}{1-\gamma} \Sigma^{\top} \rho \eta^{\top} + \hat B^{\top} 
		    &  -2 \alpha -\frac{2 \gamma}{1-\gamma} \Sigma^{\top} \rho \rho^{\top} \Sigma \\
			    \frac12 \sigma^\top a a ^\top \sigma + \frac{\gamma}{2(1- \gamma)} \eta \eta^\top +a o_2	&	- \frac{\gamma}{1-\gamma} \eta \rho^{\top} \Sigma - \hat B 
                          \end{matrix} \right)
\right)
\\
\left( \begin{matrix}
             B_{11}(t)  & B_{12}(t) \\
  	     B_{21}(t)	& B_{22}(t)
             \end{matrix}
\right)
&
= \exp \left( (T-t) \left( \begin{matrix}
                            \frac{\gamma}{1-\gamma} \Sigma^{\top} \rho \eta^{\top} + \hat B^{\top} &  -2 \alpha -\frac{2 \gamma}{1-\gamma} \Sigma^{\top} \rho \rho^{\top} \Sigma \\
			     \frac{\gamma}{2(1- \gamma)} \eta \eta^\top	&	- \frac{\gamma}{1-\gamma} \eta \rho^{\top} \Sigma - \hat B 
                          \end{matrix} \right)
\right).
\end{align*}

\end{prop}
Since the proof is very similar to the proof of Proposition \ref{prop:powerind} we omit it here.

\subsubsection{Exponential utility}
\label{sec:contexput}

In this section we want to solve the utility maximization problem for the exponential utility function
\begin{align*}
 U(x)=  -\exp(-\gamma x), \quad x \in \R,
\end{align*}
where $\gamma > 0$ denotes the risk aversion.
Note that we have already discussed this problem in the one-dimensional case in the beginning of Chapter \ref{chap:explsol} and now study it in a multivariate setting in detail.
We are also interested in pricing variance swaps which depend on the realized variance via utility indifference pricing.
We consider the case where the variance swaps are not available in the market and the initial capital $x$ is invested in the (incomplete) financial market $H$. 
For $i=1, \ldots,d$, a variance swap on the $i$-th asset of maturity $T$ is a contract which pays 
\begin{align*}
  \frac{1}{T} \int_0^T (R_{ii})_s ds 
\end{align*}
at terminal time $T$ in exchange for a previously fixed amount $K_{i}$.
That is to say the payoff of a variance swap on $H_i$ is a function of $O_T=\int_0^T R_s ds$, more precisely
\begin{align*}
 F^{i}(O_T)&=  \operatorname{Tr}( a^{ii} O_T ) -K_i,
\end{align*}
where $a^{ii}= \frac{1}{T} e^{ii}$. 
If we are only interested in the utility maximization problem without a random endowment, i.e. $F^{i}=0$, we define $a^{ii}=0$, $K_i=0$ for $i=0$.

In this section we also need a notion of admissibility.
For $\R^m_+ \times \R^n$-valued affine stochastic volatility models Vierthauer \cite{v10} shows in Theorem 3.17 that the optimal strategy in the exponential utility maximization problem is a deterministic function of time. 
Motivated by this we introduce the set $\mathcal A$ of admissible trading strategies as the set of $d$-dimensional deterministic functions of time $\pi=(\pi(t))_{t \in [0,T]}$.
This time, the trading strategy $\pi$ describes the amount of money invested in the stocks $H$ so that the number of shares is $\pi_j/H_j$ for $j= 1, \ldots, d$.
The wealth process $X^{x,\pi}$ corresponding to strategy $\pi$ and initial capital $x$ is then given by
\begin{align*}
 X_t^{x,\pi} 
&= x + \sum_{i=1}^d \int_0^t \frac{\pi_{i}(s)}{H_{i,s}} dH_{i,s} 
= x + \int_0^t \pi^{\top}(s) R_s \eta ds + \int_0^t \pi^{\top}(s) \sqrt{R_s} dQ_s.
\end{align*}

\begin{rmk}
Note that we measure the trading strategies $\pi$ in different units than in the power utility case. 
This then leads to a similar exponential structure in the process $L$.
This was also deployed in \cite{HuImkellerMuller} for example.
\end{rmk}

In the following Theorem we characterize the maximal expected utility from trading in the financial market in presence of a variance swap on the $i$-th asset
\begin{align}
\label{maxprobexp} 
V^{F^{i}}(x) = & \sup_{\pi \in \mathcal A} \E \left[-\exp \left(-\gamma \left(X_T^{x,\pi}+F^{i}(O_T)\right) \right) \right], \quad x \in \R,
\end{align}
and the optimal strategy $\pi^{F^{i}} $ for $i=0,1,\ldots,d$.

\begin{thm}
\label{thmexp}
For $i \in \{0,1, \ldots, d\}$ let $\Gamma^i$ be the solution of the ODE
\begin{align}
\label{odeexp}
-\frac{\partial \Gamma^i(t)}{\partial t} &=
 \Gamma^i(t) \left(-2\gamma \alpha + 2\gamma \Sigma^{\top} \rho \rho^{\top} \Sigma \right) \Gamma^i(t)  + B^*(\Gamma^i(t))
\\ \nonumber
& \quad  - \frac{1}{\gamma}\Gamma^i(t)\Sigma^{\top}\rho\eta^{\top} - \frac{1}{\gamma}\eta \rho^{\top} \Sigma \Gamma^i(t) +  \frac{1}{2\gamma^3}  \eta \eta^{\top}
+a^{ii},
\\
\Gamma^i(T)&=0 , \nonumber
\end{align}
for all $t \in [0,T]$.
Then the value function has the form
\begin{align*}
 V^{F^i}(x) 
&
=-\exp\left(- \gamma \left( x -K_i + \operatorname{Tr}(\Gamma^i(0)r) + \int_0^T \operatorname{Tr}(\Gamma^i(s)b) ds \right) \right), \quad x \in \R,
\end{align*}
and the optimal strategy $\pi^{F^i}$ is given by 
\begin{align}
\label{eqn:stratexp}
  \pi^{F^i}(t) = \frac{1}{\gamma^2} \eta - 2 \Gamma^i(t) \Sigma^{\top} \rho ,
\end{align}
for all $t \in [0,T]$.
\end{thm}

As before we want to use the martingale optimality principle in order to establish this Theorem.
Therefore we dynamize the problem.
For all $i=0,1, \ldots ,d,$ we define
\begin{align*}
  L^{\pi,i}_t = -\exp( -\gamma (X_t^{x,\pi} + Y^i_t)), \quad t \in [0,T], \quad \pi \in \mathcal A,
\end{align*}
where $(Y^i,Z^i)$ is the solution to
\begin{align}
\label{BSDEexp}
  Y^i_t = F^i(O_T) - \int_t^T \operatorname{Tr}((Z^i_s)^{\top} dW_s)  + \int_t^T f(R_s,Z^i_s) ds, \quad t \in [0,T].
\end{align}
The generator $f$ needs to be selected in a way such that $L^{\pi,i}$ possesses the desired properties.
This is done in the following lemma.

\begin{lem}
\label{lemsupermartexp}
For $i \in \{0,1, \ldots, d\}$, let the generator $f: S_d^+ \times M_d \to \R$ have the form
\begin{align}
\label{generatorexp}
 f(r,z^i) = - \frac{\gamma}{2} \operatorname{Tr}(z^i(z^i)^{\top}) + \frac{1}{2 \gamma} \Big| \frac{1}{\gamma}\sqrt{r} \eta - \gamma z^i \rho\Big|^2.
\end{align}
Then the solution to \eqref{BSDEexp} is given by
\begin{align}
\label{eqn:BSDEexpsol}
  Y^i_t &=  \operatorname{Tr}(\Gamma^i(t)R_t) +\operatorname{Tr} (a^{ii} O_t) -K_i + \int_t^T \operatorname{Tr}(\Gamma^i(t)b)ds,
\\ \nonumber 
  Z^i_t&=2 \sqrt{R_t} \Gamma^i(t) \Sigma^{\top},
\end{align} 
where $\Gamma^i \in S_d^+$ is the solution to \eqref{odeexp} and $O_t = \int_0^t R_s ds$.
Furthermore $L^{\pi,i}$ is a supermartingale for every strategy $\pi \in \mathcal A$ and
for
\begin{align}
\label{optstratexp}
\pi^{F^i}(t) = \frac{1}{\gamma^2} \eta - 2 \Gamma^i(t) \Sigma^{\top} \rho  ,
\end{align}
the process $L^{\pi^{F^i},i}$ is a martingale.
\end{lem}

\begin{proof}
Fix $i \in \{0,1, \ldots, d\}$ and define
\begin{align*}
  c_{zz}= \frac{\gamma}{2} ( \rho \rho^{\top} - I_d), \quad
  c_{z \sqrt x}= - \frac{1}{2\gamma} \rho \eta^{\top} , \quad
  c_x= \frac{1}{2 \gamma^3} \eta \eta^{\top} + a^{ii}. \quad
\end{align*}
Note that $c_{zz}$ is negative definite.
Indeed, if $\rho=0$, we have $c_{zz}= -\frac{\gamma}{2} I_d \in S_d^{--}$.
If $\rho \neq 0$, we know that $I_d - \rho \rho^{\top}$ is the inverse of the positive definite matrix $I_d + \frac{1}{1-\rho^{\top} \rho} \rho \rho^{\top}$ and hence is itself positive definite.
The conclusion is that $c_{zz} \in S_d^{--}$.
Then by Proposition \ref{christaeberhardriccati} there exists a unique solution $\Gamma^i \in S_d^+$. 
This allows us to find solution \eqref{eqn:BSDEexpsol} via Theorem \ref{thm:matrixBSDEsolution}.

Fix $\pi \in \mathcal A$.
By It\^o's formula we see that $L^{\pi,i}$ can be described by the product of the local martingale
\begin{align*}
  M_t^{\pi,i}= - L_0^{\pi,i} \mathcal E \left( -\gamma \left( \int_0^t \pi^{\top}(s) \sqrt{R_s} dQ_s - \int_0^t \operatorname{Tr}((Z^i_s)^{\top} dW_s)\right)\right),
\end{align*}
and the bounded variation process
\begin{align*}
 A_t^{\pi,i}
&= - \exp \left( \int_0^t \left( -\gamma \pi^{\top}(s) R_s \eta + \gamma f(R_s,Z^i_s) + \frac12 \gamma^2 |\sqrt{R_s} \pi(s) \rho^{\top} + (Z^i_s)^{\top}|^2 
\right. \right.
\\
& \hspace{2cm} \left. \left. + \frac12 \gamma^2 |\sqrt{1- \rho^{\top} \rho} \pi^{\top}(s)\sqrt{R_s} |^2   \right) ds \right).
\end{align*}
Theorem \ref{thm:martingale} implies that $M^{\pi,i}$ is a true martingale.
The process $A^{\pi,i}$ is non-increasing, if  
\begin{align*}
  & -\gamma \pi^{\top}(s) R_s \eta + \gamma f(R_s,Z^i_s) + \frac12 \gamma^2 |\sqrt{R_s} \pi(s) \rho^{\top} + (Z^i_s)^{\top}|^2 
+ \frac12 \gamma^2 |\sqrt{1- \rho^{\top} \rho} \pi^{\top}(s)\sqrt{R_s} |^2 
\geq 0.
\end{align*}
for all $s \in [0,T]$.
This is equivalent to
\begin{align}
\label{eqn:help2}
 -f(R_t,Z^i_t)
&\leq  
 \frac{\gamma}{2} \operatorname{Tr}(Z^i_t (Z^i_t)^{\top}) + \gamma \operatorname{Tr}( \sqrt{R_t} \pi(t) \rho^{\top} (Z^i_t)^{\top}) - \frac{1}{\gamma} \pi(t)^{\top} R_t \eta + \frac{\gamma}{2} |\pi^{\top} \sqrt{R_t}|^2
\\ \nonumber
&=
\frac{\gamma}{2} \left|\sqrt{R_t} \pi(t) - \left(\frac{1}{\gamma^2 }\sqrt{R_t} \eta - Z^i_t \rho\right)\right|^2 - \frac{1}{2\gamma} \left(\frac{1}{\gamma} \sqrt{R_t}\eta - \gamma Z^i_t \rho\right)^2  
+ \frac{\gamma}{2} \operatorname{Tr}(Z^i_t (Z^i_t)^{\top}),
\end{align}
which holds true by formula \eqref{generatorexp}.
Hence $A^{\pi,i}$ is non-increasing and $L^{\pi,i}=M^{\pi,i}A^{\pi,i}$ is a supermartingale.
From \eqref{eqn:help2} we see in particular that $A^{\pi^{F^i},i}=-1$ is constant and thus $L^{\pi^{F^i},i}=-M^{\pi^{F^i},i}$ is a true martingale.
\end{proof}

\begin{proof}[Proof of Theorem \ref{thmexp}]
 Follows by the same reasoning as for Theorem \ref{thmpower}.
\end{proof}

For all $i \in \{1, \ldots,d\}$, the indifference price of the variance swap 
$$F^i(O_T)=  \operatorname{Tr}\left( a^{ii} O_T \right) -K_{i}$$ 
on the $i$-th asset is defined as the value $p^i$ for which the investor is indifferent between buying $F^i$ for the amount $p^i$ and receiving a random income $F^i$ at terminal time $T$ or not having it, i.e. 
\begin{align*}
 V^{F^i}(x-p^i) = V^0(x),
\end{align*}
for all $x \in \R$.
The optimal strategy $\pi^{F^i}$ which attains the maximal expected utility in the presence of $F^i$ can be decomposed into a sum of a pure investment part $\pi^0$ and a hedging component $\Delta^i$, i.e.
\begin{align}
\label{eqn:opthedge}
  \pi^{F^i}(t)=\pi^0(t) + \Delta^i(t), \quad t \in [0,T].
\end{align}
We therefore call $\Delta^i$ the optimal hedge.

\begin{prop}
For $i \in \{1, \ldots, d \}$ the indifference price $p^i$ and the optimal hedge $\Delta^i$ of $F^i(O_T)$ are explicitly given by
 \begin{align*}
  p^i&= -K_{i} + \operatorname{Tr}((\Gamma^{i}(0)-\Gamma^0(0))r) + \int_0^T \operatorname{Tr}((\Gamma^{i}(s)-\Gamma^0(s))b) ds,
\\
\Delta^i(t) &= 2(\Gamma^{i}(t)-\Gamma^0(t)) \Sigma^{\top} \rho, \quad t \in [0,T],
 \end{align*}
where $\Gamma^{i}$ and $\Gamma^0$ are the solutions of \eqref{odeexp}. 

\end{prop}

\begin{proof}
Fix $i \in \{1, \ldots, d\}$ and recall the value functions 
\begin{align*}
 V^{F^i}(x-p^i) &= - \exp \left( -\gamma \left(x-p^i-K_{i} +\operatorname{Tr}(\Gamma^{i}(0)r) + \int_0^T \operatorname{Tr}(\Gamma^i(s) b) ds \right)\right), \\
V^0(x) & = - \exp \left( -\gamma \left(x+\operatorname{Tr}(\Gamma^0(t)r) + \int_0^T \operatorname{Tr}(\Gamma^0(s) b) ds \right)\right),
\end{align*}
from Theorem \ref{thmexp}.
Equating them immediately gives the first part of the result.
The second part then follows from \eqref{eqn:stratexp} and \eqref{eqn:opthedge}.
\end{proof}

\subsection{Solution in a multivariate affine stochastic volatility model with jumps}
\label{sec:BNS}

We now consider a model with jumps which is a natural multivariate extension of the model of \cite{bns01} and has been applied e.g. in optimal portfolio selection, see \cite{bkr03} and the references therein.
As before the asset price process $H$ is modeled as stochastic exponential $H=H_0 \mathcal E (N)$ with 
\begin{align}
\label{eqn:defN}
  dN_t&=  R_{t} \eta dt + \sqrt{R_{t}} dQ_t, \quad t \in [0,T],
\end{align}
where $Q$ is a $d$-dimensional vector Brownian motion and $\eta$ a constant parameter.
By $R$ we denote the Ornstein-Uhlenbeck-type stochastic process with dynamics
\begin{align}
\label{eqn:RBNS}
 dR_t & =  (\lambda + \lambda( R_{t})) dt + d J_t, 
\end{align}
and a starting value $R_0=r$.
Here $\lambda \in S^+_d$ and $\Lambda: S_d \to S_d$ is the linear map $\Lambda(r)=\sum_{i,j} \beta^{ij}r_{ij}$ with $\beta^{ij}=\beta^{ji} \in S_d$ and such that $\operatorname{Tr}(\Lambda(r)x) \geq 0$ for all $r,x \in S_d^+$ with $\operatorname{Tr}(rx)=0$.
We denote its adjoint operator by $\Lambda^*$.
The process $J$ is an independent affine 
process with admissible parameter set $(0,b^J,0,m^J,0)$, starting at $0$.
Our goal is again to maximize the expected terminal wealth from trading in the market. 

\begin{prop}
The process $(R,N)$ is a multivariate stochastic volatility model with functions $\Phi$ and $\Psi$ solving
\begin{align}
\label{eqn:affwjumps}
  \frac{\partial \Phi(t,u,v)}{\partial t}
&=(\lambda + b^J)  \Psi(t,u,v) - \int_{S_d^+ \setminus \{0\}} (e^{\operatorname{Tr}(\xi \Psi(t,u,v))}-1) m^J(d\xi),
\quad \Phi(0,u,v)=0,
\\ \nonumber
  \frac{\partial \Psi(t,u,v)}{\partial t}
  &=
\Lambda^*(\Psi(t,u,v)) + \frac12 v v^{\top},
\quad \Psi(0,u,v)=u,
\end{align}
for all $(t,u,v) \in \mathcal Q$.
\end{prop}

\begin{proof}
By construction $(R,N)$ is again stochastically continuous.
Moreover it follows that $H$ is a martingale under the probability measure $d \mathbb Q=\mathcal E (- \int_0^{\cdot} \eta^{\top} \sqrt{R_s} dQ_s )$ as in the proof of Proposition \ref{prop:contstochvol}.
\\
\indent
Again similarly to Proposition \ref{prop:contstochvol} we have that the conditional Fourier-Laplace transform for $(t,u,v) \in \mathcal Q$
\begin{align*}
  h(t-s,R_s,N_s) &=\E^{\mathbb Q} \left[ e^{\operatorname{Tr} \left( u R_t\right) + v^\top  N_t} | (R_s,N_s)\right],
\end{align*}
is a martingale.
We assume that the conditional Fourier-Laplace transform is of exponentially affine form, more precisely that there exist functions $\Psi:\mathcal Q \to S_d+\mathrm{i}S_d$ and $\Phi:\mathcal Q \to \mathbb C^d$ such that $h$ has the form
\begin{align*}
  h(t-s,R_s,N_s)&=
\exp \left(  \operatorname{Tr}\left( \Psi(t-s,u,v) R_s \right) + v^\top N_s + \Phi(t-s,u,v) \right).
\end{align*}
We apply It\^o's formula and stipulate that the bounded variation term needs to be zero.
More precisely, this is equivalent to the equation
\begin{align*}
  0&=
\operatorname{Tr} \l( \frac{\Psi(t-s,u,v)}{\partial s} R_s \r) + \frac{\Phi(t-s,u,v)}{\partial s}
+ \operatorname{Tr} \l( (\lambda + \Lambda(R_s) + b^J) \Psi(t-s,u,v) \r)
+ \frac12 v^{\top} R_s v
\\
& \quad
+ \int_{S_d^+ \setminus \{0\}} ( e^{\operatorname{Tr}(\xi \Psi(t-s,u,v))}-1) m^J(d\xi),
\end{align*}
for $s \leq t$.
Equating coefficients leads to the following system of ODEs
\begin{align*}
  -\frac{\Psi(t-s,u,v)}{\partial s} 
&=
\Lambda^*(\Psi(t-s,u,v))
+\frac12 v^{\top} v,
\quad \Psi(t,u,v)=u,
\\
-\frac{\Phi(t-s,u,v)}{\partial s}
&=
(\lambda +b^J)\Psi(t-s,u,v)
+ \int_{S_d^+ \setminus \{0\}} ( e^{\operatorname{Tr}(\xi \Psi(t-s,u,v))}-1) m^J(d\xi),
\quad \Phi(t,u,v)&=0.
\end{align*}
With a change of variable we see that the above ODEs coincide with \eqref{eqn:affwjumps}.
The ODE for $\Phi$ can be solved via integration, provided there exists a solution to
\begin{align*}
  \frac{\partial \Psi(t,u,v)}{\partial t} 
&= 
\Lambda^*(\Psi(t,u,v)) + \frac12 v v^{\top},
 \quad 
\Psi(0,u,v)=u,  \quad (t,u,v) \in \mathcal Q.
\end{align*}
This is a linear ODE, hence by e.g. \cite{br89}, there exists a unique solution $\Psi(t,u,v) \in S_d$ for all $(t, u,v) \in \mathcal Q$.
This implies the result.
\end{proof}

\subsubsection{Power utility}

The investor's utility function is assumed to be
 \begin{align*}
  U(x) = \frac{1}{\gamma} x^{\gamma}, \quad \quad x \geq 0, \ \gamma \in (0,1),
 \end{align*}
and we let $F=0$.
By $\mathcal A$ we denote the set of all $d$-dimensional predictable processes $\pi$ that satisfy a.s. $\int_0^T \pi_s^{\top} \pi_s ds < \infty$.
For $i=1, \ldots, d$,  $\pi_i$ again denotes the fraction of the wealth invested in stock $i$ and any process $\pi \in \mathcal A$ is called an admissible (trading) strategy.
Hence, for a trading strategy $\pi$ and initial capital $x$ the wealth process has dynamics
\begin{align*}
 X_t^{x,\pi} = x+ \int_0^t X_{s}^{x, \pi} \pi_{s}^{\top} dN_s = x + \int_0^t X_s^{x, \pi} \pi_s^{\top} R_s \eta ds + \int_0^t X_{s}^{x, \pi} \pi_{s}^{\top} \sqrt{R_{s}} d Q_s, 
\end{align*}
for $t \in [0,T]$ which can also be written as a stochastic exponential
\begin{align*}
 X_t^{x,\pi} = x \mathcal E \left( \int_0^t \pi^{\top}_s R_s \eta ds + \int_0^t \pi^{\top}_s \sqrt{R_{s}} d Q_s \right), \quad t \in [0,T].
\end{align*}
The investor wants to maximize their expected utility of terminal wealth, i.e. we search for the value function
\begin{align*}
 V(x)= \sup_{\pi \in \mathcal A} \E \left[ \frac{1}{\gamma} (X_T^{x,\pi})^{\gamma}  \right]
, \quad x \geq 0.
\end{align*}
We are able to describe the value function and the optimal strategy of the maximization problem in terms of an ODE.
\begin{thm}
\label{thm:BNSpower}
Suppose the jump measure $m^J$ satisfies 
\begin{align*}
\int_{|\operatorname{Tr}(\Gamma(t) \xi)| > 1} e^{-\operatorname{Tr}(\Gamma(t) \xi)} m^J(d\xi) < \infty, \quad t \in  [0,T],
\end{align*}
where $\Gamma$ is the solution of the ODE
\begin{align}
\label{eqn:odeBNSpower}
- \frac{d \Gamma(t)}{d t} 
&=
\Lambda^*(\Gamma(t)) - \frac{\gamma}{2(1-\gamma)} \eta \eta^{\top}, \quad 
\Gamma(T)=0 .
\end{align}
Then the value function is given by
\begin{align*}
 V(x) 
&=\frac{1}{\gamma} x^{\gamma} 
\exp\left(-\operatorname{Tr}(\Gamma(0) r) - \int_0^T \operatorname{Tr}(\Gamma(s)(b^J+\lambda)) ds
- \int_0^T \int_{S_d^+\setminus \{0\}} (e^{-\operatorname{Tr}(\Gamma(s) \xi) } - 1 ) m^J (d\xi) ds \right), 
\end{align*}
for $x \geq 0,$
and the optimal strategy $\pi^{opt}$ is 
\begin{align*}
  \pi_t^{opt} \equiv \frac{1}{1-\gamma} \eta, \quad t \in [0,T].
\end{align*}
\end{thm}

As before we solve the problem using the martingale optimality principle.
Applying the utility function to $X^{x, \pi}$ we get
\begin{align*}
 \frac{1}{\gamma} (X_t^{x, \pi})^{\gamma} 
= \frac{1}{\gamma} x^{\gamma} \exp \left( \int_0^t \gamma \pi_s^{\top} R_s \eta ds + \int_0^t \gamma \pi^{\top}_s \sqrt{R_{s}} d Q_s -\frac12 \int_0^t \gamma \pi^{\top}_s R_s \pi_s ds \right),
\end{align*}
for $t \in [0,T]$.
This suggests the following choice of $L^{\pi}$
\begin{align*}
 L_t^{\pi} = x^{\gamma} \exp \left( \int_0^t \gamma \pi_s^{\top} R_s \eta ds + \int_0^t \gamma \pi^{\top}_{s} \sqrt{R_{s}} d Q_s -\frac12 \int_0^t \gamma \pi^{\top}_s R_s \pi_s ds - Y_t \right), 
\end{align*}
for $t \in [0,T]$, and where $Y$ is the first component of the solution of a BSDE with terminal condition 0.
More precisely we want to find a generator $f$ for the BSDE
\begin{align}
\label{BSDEBNSpower} 
Y_t & = 0 - \int_t^T \int_{S_d^+ \setminus \{0\}} K_s(\xi) (\mu^J(ds,d\xi)-m^J(d\xi)ds) + \int_t^T f(R_s,K_s) ds,
\end{align}
$t \in [0,T]$, such that with its solution $(Y,K)$ $L^{\pi}$ satisfies the above requirements.

\begin{lem}
\label{lem:supermartBNSpower}
Let the jump measure $m^J$ satisfy
\begin{align}
\label{eqn:integrability2}
\int_{|\operatorname{Tr}(\Gamma(t) \xi)| > 1} e^{-\operatorname{Tr}(\Gamma(t) \xi)} m^J(d\xi) < \infty, \quad t \in  [0,T],
\end{align}
where $\Gamma$ is the solution of \eqref{eqn:odeBNSpower}.
Suppose the generator in \eqref{BSDEpower} is of the following form
\begin{align}
\label{generatorBNSpower}
 f(r,k) & = 
- \frac{\gamma}{2(1-\gamma)} \eta^{\top} r \eta
- \int_{S_d^+ \setminus \{0\}} \left( e^{-k(\xi)} -1 +k(\xi) \right) m^J(d\xi) , 
\end{align}
for all $r \in S_d^+$ and $k:S_d^+ \to \R$.
Then BSDE \eqref{BSDEBNSpower} is solved by 
\begin{align*}
  Y_t & = \operatorname{Tr}(\Gamma(t) R_t) + \int_t^T \operatorname{Tr}(\Gamma(s)(b^J+\lambda)) ds
+ \int_t^T \int_{S_d^+\setminus \{0\}} (1-e^{-\operatorname{Tr}(\Gamma(s) \xi) } ) m^J (d\xi)ds
\\
  K_t(\xi) & = \operatorname{Tr}(\Gamma(t) \xi), \quad t \in [0,T], \quad \xi \in S_d^+.
\end{align*}
Moreover $L^{\pi}$ is a supermartingale for every strategy $\pi \in \mathcal A$ and if $\pi^{opt}$ satisfies
\begin{align}
\pi^{opt}_t & = \frac{1}{1- \gamma} \eta, \quad t \in [0,T],
\end{align}
then $L^{\pi^{opt}}$ is a martingale.
\end{lem}

\begin{proof}
For all $y \in \R$ we define
\begin{align*}
c_x = -\frac{\gamma}{2(1-\gamma)} \eta \eta^{\top}, \quad g_t(y)=-e^{-y}+1-y.
\end{align*}
Then we can see that by Corollary \ref{cor:christaeberhardriccati} there exists a unique solution $\Gamma$ with values in $S_d^-$ to \eqref{eqn:odeBNSpower}.
As a result we find the above solution of \eqref{BSDEBNSpower} with Theorem \ref{thm:matrixBSDEsolution}.

We apply It\^o's formula which gives that for all $\pi \in \mathcal A$
\begin{align*}
  dL^{\pi}_t 
&=
L^{\pi}_t  \left( \gamma \pi^{\top}_t \sqrt{R_t} dQ_t\right)\
+ L^{\pi}_t \left( \gamma \pi_t R_t \eta - \frac12 \gamma \pi^{\top}_t R_t \pi_t + f(R_t,K_t) + \frac12 \gamma^2 \pi^{\top}_t R_t \pi_t \right) dt
\\
& \quad 
+ L_t^{\pi} \int_{S_d^+ \setminus \{0\} } (e^{-K_t(\xi)}-1) \left( \mu^J(dt,d\xi) - m^J(d\xi) dt\right)
\\
& \quad
+ L_t^{\pi} \int_{S_d^+ \setminus \{0\} } \left( e^{-K_t(\xi)}-1 + K_t(\xi)\right) m^J(d\xi)dt,
\end{align*}
where we have used integrability condition \eqref{eqn:integrability2}. 
This means that $L^{\pi}$ is a local supermartingale for all $\pi \in \mathcal A$, if the finite variation part $dt \otimes \mathbb P$-a.e. satisfies
\begin{align*}
&  L^{\pi}_t \left( \gamma \pi^{\top}_t R_t \eta - \frac12 \gamma \pi^{\top}_t R_t \pi_t + f(R_t,K_t) + \frac12 \gamma^2 \pi^{\top}_t R_t \pi_t
\right.
\\
 & \left. \qquad
+ \int_{S_d^+ \setminus \{0\} } \left( e^{-K_t(\xi)}-1 + K_t(\xi)\right) m^J(d \xi)  \right) \leq 0.
\end{align*}
Since $L^{\pi}>0$, the generator $f$ needs to fulfill
\begin{align*}
f(R_t,K_t) 
& \leq
 - \gamma \pi_t^{\top} R_t \eta 
+ \frac12 \gamma (1-\gamma)\pi^{\top}_t R_t \pi_t 
- \int_{S_d^+ \setminus \{0\} } \left( e^{-K_t(\xi)}-1 + K_t(\xi)\right) m^J(d \xi)  ,
\end{align*}
which is equivalent to
\begin{align*}
f(R_t,K_t) 
& \leq
\frac12 \gamma (1-\gamma) \big|\sqrt{R_t} \pi_t - \frac{1}{1-\gamma} \sqrt{R_t} \eta \big|^2
 - \frac{\gamma}{2(1-\gamma)} |\sqrt{R_t} \eta|^2
\\
& \quad
- \int_{S_d^+ \setminus \{0\} } \left( e^{-K_t(\xi)}-1 + K_t(\xi)\right) m^J(d \xi)  .
\end{align*}
With \eqref{generatorBNSpower} this inequality is true for all $\pi \in \mathcal A$ and hence $L^{\pi}$ a local supermartingale.
Obviously the inequality is an equality for $\pi^{opt}$, for which $L^{\pi^{opt}}$ is then a local martingale.

Since $L^{\pi}$ is bounded below by $0$, the fact that $L^{\pi}$ is a supermartingale for all $\pi \in \mathcal A$ follows as in the proof of Lemma \ref{lem:supermartpower} by Fatou's Lemma.
Note that the process $L^{\pi^{opt}}$ is given by
\begin{align*}
L_t^{\pi^{opt}} &= x^{\gamma} \mathcal E \left( \int_0^t \gamma (\pi^{opt}_s)^{\top} \sqrt{R_s} dQ_s
+ \int_0^t \int_{S_d^+ \setminus \{0\} } (e^{-K_s(\xi)}-1) (\mu^J(ds,d\xi) - m^J(d\xi) ds)
\right),
\end{align*}
and hence, setting $\sigma_Q(s) \equiv \frac{\gamma}{1-\gamma} \eta$, $\sigma_W(s) \equiv 0$ and $\sigma_{\mu}(s)=-\Gamma(s)$, $s \in [0,T]$, we obtain the martingale property of $L^{\pi^{opt}}$ by Theorem \ref{thm:martingale}.
\end{proof}

\begin{proof}[Proof of Theorem \ref{thm:BNSpower}]
Note that we have from Lemma \ref{lem:supermartBNSpower} 
$$\mathbb E\left[U(X_T^{x,\pi})\right] 
=\mathbb E\left[\frac{1}{\gamma} L_T^{\pi} \right] 
\leq \mathbb E \left[ \frac{1}{\gamma} L_0^{\pi} \right] 
= \frac{1}{\gamma} x^{\gamma} \exp(-Y_0) , \quad \pi \in \mathcal A.
$$
Due to $L^{\pi^{opt}}$ being a martingale we have 
$$\mathbb E \l[U\l(X_T^{x,\pi^{opt}} \r) \r] = \mathbb E \l[ \frac{1}{\gamma} L_0^{\pi^{opt}} \r] $$ and thus \eqref{optstratpower} is indeed the optimal strategy. 
This also provides the representation of the value function.
\end{proof}

\subsubsection{Exponential utility}

We now show how to solve the utility maximization problem for exponential utility in the presence of jumps and random revenues $F$.
The exponential utility function is given by
\begin{align*}
  U(x)=-\exp(-\gamma x), \quad x \in \R,
\end{align*}
where $\gamma>0$ denotes the risk aversion.
As before in Section \ref{sec:contexput} the deterministic $d$-dimensional functions $\pi=(\pi(t))_{t \in [0,T]}$ form the set of admissible strategies $\mathcal A$.
For $i = 1, \ldots ,d$, $\pi^i$ denotes again the amount of money invested in stock $H^i$, where we recall that $H=(H^1, \ldots, H^d)= \mathcal E (N)$ with $N$ defined in \eqref{eqn:defN}.
In particular the wealth process $X^{x,\pi}$ corresponding to a trading strategy $\pi$ and an initial capital $x \geq 0$ satisfies
\begin{align*}
 X_t^{x,\pi} 
&= x + \sum_{i=1}^d \int_0^t \frac{\pi_{i}(s)}{H_{i,s}} dH_{i,s} 
= x+ \int_0^t \pi^{\top}(s) dN_s 
\\
&= x + \int_0^t \pi^{\top}(s) R_s \eta ds + \int_0^t \pi^{\top}(s) \sqrt{R_{s}} d Q_s.
\end{align*}
As in Section \ref{sec:contexput}, using the same notation, we will also compute the utility indifference prices for variance swaps.
This means that we need to solve the problem
\begin{align*}
  V^{F^i}(x) &= \sup_{\pi \in A} \mathbb E \left[ U\left(X_t^{x,\pi} + F^i \right) \right], \quad x \geq 0,
\end{align*}
for $i \in \{0,1, \ldots,d\}$, which is done in the following Theorem.

\begin{thm}
\label{thm:BNSexp}
  Let $i \in \{0,1,\ldots,d\}$,
\begin{align*}
  \int_{|\gamma \operatorname{Tr}(\Gamma^i(s)\xi)|>1}  e^{\gamma \operatorname{Tr}(\Gamma^i(s)\xi)} m^J(d\xi)< \infty, \quad s \in [0,T],
\end{align*}
where $\Gamma^i$ is the solution of the ODE
\begin{align}
\label{eqn:ODEBNSexp}
  -\frac{d \Gamma^i(t)}{d t} &=
\Lambda^*(\Gamma^i(t)) + \frac{1}{2 \gamma } \eta \eta^{\top} + a^{ii},
\quad \Gamma^i(T)=0.
\end{align}
Then the value function satisfies
\begin{align*}
  V^{F^i}(x)&=-\exp \left(-\gamma \left( x-K_{i}+\operatorname{Tr}(\Gamma^i(0) r) + \int_0^T \operatorname{Tr}(\Gamma^i(s) (b^J+ \lambda)) ds 
\right. \right.
\\ \nonumber
& \hspace{2cm} \left. \left.
-  \frac{1}{\gamma} \int_0^T \int_{S_d^+ \setminus \{0\}} \left( e^{\gamma \operatorname{Tr}(\Gamma^i(s)\xi)} -1\right) m^J(d\xi) ds  \right)\right),
\end{align*}
and the optimal strategy $\pi^{F^i}$ is given by
\begin{align*}
  \pi^{F^i}(t)& \equiv \frac{1}{\gamma} \eta, \quad t \in [0,T].
\end{align*}
\end{thm}
We will use the martingale optimality principle again and construct for every $i \in \{0,1,\ldots,d\}$, a process
\begin{align*}
  L^{\pi,i}_t &= -\exp(-\gamma (X_t^{x,\pi} + Y^i_t)), \quad t \in [0,T], \quad \pi \in \mathcal A,
\end{align*}
where $(Y^i,K^i)$ is the solution of
\begin{align}
\label{eqn:BSDEBNSexp}
  Y^i_t
&=F^i
- \int_t^T \int_{S_d^+ \setminus \{0\}} K^i_s(\xi) (\mu^J(ds,d\xi)-m^J(d\xi)ds)
+ \int_t^T f(R_s,K^i_s) ds,
\end{align}
with a generator $f$ such that $L^{\pi,i}$ satisfies the conditions
\begin{itemize}
 \item the terminal condition $L^{\pi,i}_T = U(X_T^{x,\pi}+F^i)$ is satisfied for all $\pi \in \mathcal A$,
 \item the process $L^{\pi,i}$ is a supermartingale for all $\pi \in \mathcal A$ and there is a $\pi^{F^i} \in \mathcal A$ such that $L^{\pi^{F^i}}$ is a martingale.
\end{itemize}

The following Lemma shows how the generator of BSDE \eqref{eqn:BSDEBNSexp} needs to be chosen in order to meet the above requirements.

\begin{lem}
Let $i \in \{0,1, \ldots, d\}$, 
\begin{align*}
  \int_{|\gamma \operatorname{Tr}(\Gamma^i(t))|>1}  e^{\gamma \operatorname{Tr}(\Gamma^i(t)\xi)} m^J(d\xi)< \infty,
\end{align*}
for all $t \in [0,T]$ and with $\Gamma^i$ being the solution of \eqref{eqn:ODEBNSexp}.
Let the generator $f$ in \eqref{eqn:BSDEBNSexp} have the form
\begin{align}
\label{eqn:generatorBNSexp}
  f(r,k^i)&= \frac{1}{2\gamma} \eta^{\top} r \eta - \frac{1}{\gamma} \int_{S_d^+ \setminus \{0\}} \left( e^{\gamma k^i(\xi)} -1+\gamma k^i(\xi)\right) m^J(d\xi),
\end{align}
for all $r \in S_d^+$ and $k^i:S_d^+ \to \R$.
Then the solution of BSDE \eqref{eqn:BSDEBNSexp} is given by
\begin{align}
\label{eqn:solBSDEBNS}
  Y^i_t &= \operatorname{Tr}(\Gamma^i(t) R_t) + \operatorname{Tr} \left( a^{ii} \int_0^t R_s ds \right) -K_{i} + \int_t^T \operatorname{Tr}(\Gamma^i(s) (b^J+ \lambda)) ds 
\\ \nonumber
& \quad
- \frac{1}{\gamma} \int_t^T \int_{S_d^+ \setminus \{0\}} \left(e^{\gamma \operatorname{Tr}(\Gamma^i(s)\xi)} - 1\right) m^J(d\xi) ds,
  \\ \nonumber
  K^i_t(\xi) &=\operatorname{Tr}(\Gamma^i(t)\xi), \quad t \in [0,T], \ \xi \in S_d^+.
\end{align}
Moreover for all $\pi \in \mathcal A$ the process $L^{\pi,i}$ is a supermartingale and $L^{\pi^{F^i},i}$ is a martingale.
\end{lem}

\begin{proof}
Fix $i \in \{0,1, \ldots, d\}$ and define for all $y \in \R$
\begin{align*}
  g_x= \frac{1}{2 \gamma} \eta \eta^{\top} + a^{ii}, \quad g_t(y)=-\frac{1}{\gamma} \left(e^{\gamma y}-1 + \gamma y \right).
\end{align*}
Using Proposition \ref{christaeberhardriccati} we see that there exists a unique solution $\Gamma^i \in S_d^+$ of \eqref{eqn:ODEBNSexp}.
This implies \eqref{eqn:solBSDEBNS} by Theorem \ref{thm:matrixBSDEsolution}.

Fix $\pi \in  \mathcal A$.
Note that $L^{\pi,i}$ can be written as a product $M^{\pi,i} V^{\pi,i}$ of the two processes
\begin{align*}
  M_t^{\pi,i} &= -L_0^{\pi,i} \mathcal E 
\Biggl( -\gamma \int_0^t \pi^{\top}(s) \sqrt{R_s} dQ_s 
+ \int_0^t \int_{S_d^+ \setminus \{0\}} \left( e^{\gamma K^i_s(\xi)} -1 \right) (\mu^R(ds, d\xi)-m^J(d\xi)ds 
\Biggr)
,
\\
V_t^{\pi,i} &= - \exp \left( \int_0^t \left( -\gamma \pi^{\top}(s) R_s \eta + \gamma f(R_s,K^i_s) + \frac12 \gamma^2 \pi^{\top}(s) R_s \pi_s 
\right. \right.
\\
& \hspace{2cm} \left. \left. 
+ \int_{S_d^+ \setminus \{0\}} \left( e^{\gamma K^i_s(\xi)} -1 + \gamma K^i_s(\xi)\right) m^J(d\xi)  \right) ds \right).
\end{align*}
Setting $\sigma_Q(s)=-\gamma \pi_s$, $\sigma_W(s)\equiv 0$ and $\sigma_{mu}(s)=\gamma \Gamma(s)$, we have from Theorem \ref{thm:martingale} that $M^{\pi,i}$ is a true martingale.
In order for $V^{\pi,i}$ to be decreasing, it needs to be ensured that
\begin{align}
\label{eqn:help1} 
-\gamma \pi^{\top}(s) R_s \eta + \gamma f(R_s,K^i_s) + \frac12 \gamma^2 \pi^{\top}(s) R_s \pi_s  
+ \int_{S_d^+ \setminus \{0\}} \left( e^{\gamma K^i_s(\xi)} -1 + \gamma K^i_s(\xi)\right) m^J(d\xi)
\geq 0, 
\end{align}
$ds \otimes \mathbb P$-a.e.
Taking formulas \eqref{eqn:generatorBNSexp} and \eqref{eqn:solBSDEBNS} into account this is indeed true, since \eqref{eqn:help1} is equivalent to
\begin{align*}
-f(R_t,K^i_t) 
&\leq
 \frac12 \gamma |\pi(t)^{\top}\sqrt{R_t} - \frac{1}{\gamma}\eta^{\top} \sqrt{R_t}|^2 
-\frac{1}{2\gamma} \eta^{\top} R_t \eta
\\
& \quad
+\int_{S_d^+ \setminus \{0\}} \frac{1}{\gamma}\left( e^{\gamma K^i_t(\xi) } - 1 + \gamma K^i_t(\xi)\right) m^J(d\xi).
\end{align*}
Since $M^{\pi,i}$ is a martingale and $V^{\pi,i}$ is non-increasing, $L^{\pi,i}=M^{\pi,i}V^{\pi,i}$ is a supermartingale.
It is straightforward that $V^{\pi^{F^i},i}_s=-1$ for $s \in [0,T]$ and thus $L^{\pi^{F^i},i}=-M^{\pi^{\mc F^i},i}$ is a true martingale.
\end{proof}

\begin{proof}[Proof of Theorem \ref{thm:BNSexp}]
  The proof follows the same reasoning as the proof of Theorem \ref{thm:BNSpower}.
\end{proof}

Recall that for $i \in \{1, \ldots,d \}$ the indifference price of the variance swap $F^i$ on the $i$-th asset is the value $p^i$ 
such that for all $x \in \R$ the value
$ V^{F^i}(x-p^i)$ equals $V^0(x)$.

\begin{prop}
For $i \in \{1, \ldots,d \}$ the indifference price $p^i$ 
is explicitly given by
 \begin{align*}
  p^i&= - K_{i} +\operatorname{Tr}((\Gamma^i(0)- \Gamma^0(0))r) + \int_0^T \operatorname{Tr}((\Gamma^i(s)-\Gamma^0(s))( b^J+\lambda) ) ds
\\
& \quad
  -\frac{1}{\gamma} \int_0^T \int_{S_d^+ \setminus \{0\}} \left( e^{\gamma \operatorname{Tr}(\Gamma^i(s)\xi)} - e^{\gamma \operatorname{Tr}(\Gamma^0(s)\xi)} \right) m^J(d\xi),
 \end{align*}
where $\Gamma^i$ and $\Gamma^0$ are the respective solutions of 
\begin{align*}
  \frac{\partial \Gamma^i(t)}{\partial t} &= \Lambda^{*}(\Gamma^i(t))+ \frac{1}{2 \gamma} \eta \eta^{\top} + a^{ii}, \quad \Gamma^i(T)=0,
\\
  \frac{\partial \Gamma^0(t)}{\partial t} &= \Lambda^{*}(\Gamma^0(t))+ \frac{1}{2 \gamma} \eta \eta^{\top} , \quad \Gamma^0(T)=0.
\end{align*}
\end{prop}

\begin{proof}
For $i \in \{0, 1, \ldots,d\}$ it follows from Theorem \ref{thm:BNSexp} that the value functions have the form
\begin{align*}
V^{F^i}(x-p^i) 
& = - \exp \left( -\gamma \left(x-p^i - K_{i} +\operatorname{Tr}(\Gamma^i(0)r) + \int_0^T \operatorname{Tr}(\Gamma^i(s)( b^J+\lambda) ) ds
\right. \right.
\\
& \hspace{2cm} \left. \left.
   -\frac{1}{\gamma} \int_0^T \int_{S_d^+ \setminus \{0\}} \left( e^{\gamma \operatorname{Tr}(\Gamma^i(s)\xi)} - 1 \right) m^J(d\xi) \right)\right).
\end{align*}
Equating $V^{F^i}(x-p^i)$ and $V^0(x)$ for $i=1, \ldots,d$, immediately gives the result.
\end{proof}

\bibliography{bibliography}
\bibliographystyle{plain}

\end{document}